\documentclass{amsart}[11pt]

\usepackage{etex}
\usepackage[lmargin=1in,rmargin=1in,tmargin=1in,bmargin=1in]{geometry}
\usepackage{amsmath,amsthm,amsfonts,amssymb,verbatim}
\usepackage{graphicx}
\usepackage{tabularx}
\usepackage{hyperref}
\usepackage{xcolor}
\usepackage{booktabs}
\usepackage{pinlabel}
\usepackage{picins}
\usepackage[font=footnotesize,labelfont=bf]{caption}
\usepackage{centernot}

\def\co{\colon\thinspace}

\DeclareMathAlphabet{\mathpzc}{OT1}{pzc}{m}{it}
\DeclareMathOperator{\rk}{rk}

\newcommand{\Z}{\mathbb{Z}}
\newcommand{\Q}{\mathbb{Q}}
\newcommand{\R}{\mathbb{R}}
\newcommand{\F}{\mathbb{F}}

\newcommand{\Zp}{\Z/p\Z}
\newcommand{\Ztwo}{\Z/2\Z}

\newcommand{\HFhat}{\widehat{\mathit{HF}}}
\newcommand{\HFplus}{\mathit{HF}^{+}}

\newcommand{\CFKminus}{\mathit{CFK}^-}
\newcommand{\CFKinfty}{\mathit{CFK}^\infty}

\newcommand{\HFKhat}{\widehat{\mathit{HFK}}}

\newcommand{\CFD}{\widehat{\mathit{CFD}}}

\newcommand{\spin}{\mathfrak{s}}
\newcommand{\spinc}{\operatorname{Spin}^c}

\newcommand{\cylinder}{\overline{T}}
\newcommand{\pcylinder}{\cylinder_\bullet}
\newcommand{\torus}{T}
\newcommand{\ptorus}{\torus_\bullet}
\newcommand{\plane}{\widetilde{T}}
\newcommand{\pplane}{\plane_\bullet}
\newcommand{\curveset}{\boldsymbol{\gamma}}
\newcommand{\hatGamma}{\widehat\Gamma}

\newcommand{\gr}{\operatorname{gr}}
\newcommand{\mrel}{M_{rel}}
\newcommand{\mrelsum}{\Sigma\mrel}

\newcommand{\boldtau}{\boldsymbol{\tau}}

\newtheorem{theorem}{Theorem}

\newtheorem{corollary}[theorem]{Corollary}
\newtheorem{proposition}[theorem]{Proposition}
\newtheorem{lemma}[theorem]{Lemma}
\newtheorem{conjecture}[theorem]{Conjecture}

\theoremstyle{definition}
\newtheorem{definition}[theorem]{Definition}
\newtheorem{question}[theorem]{Question}

\newtheorem{remark}[theorem]{Remark}
\newtheorem{example}[theorem]{Example}

\title{Heegaard Floer homology and cosmetic surgeries in $S^3$}

\author[Jonathan Hanselman]{Jonathan Hanselman}
\address {Department of Mathematics, Princeton University.\newline \it{E-mail address:} \tt{jh66@princeton.edu}}
\thanks{The author was partially supported by NSF grant DMS-1812527 }

\begin{document}
\maketitle

\begin{abstract} If a knot $K$ in $S^3$ admits a pair of truly cosmetic surgeries, we show that the surgery slopes are either $\pm 2$ or $\pm 1/q$ for some value of $q$ that is explicitly determined by the knot Floer homology of $K$. Moreover, in the former case the genus of $K$ must be two, and in the latter case there is bound relating $q$ to the genus and the Heegaard Floer thickness of $K$. As a consequence, we show that the cosmetic crossing conjecture holds for alternating knots (or more generally, Heegaard Floer thin knots) with genus not equal to two. We also show that the conjecture holds for any knot $K$ for which each prime summand of $K$ has at most 16 crossings; our techniques rule out cosmetic surgeries in this setting except for slopes $\pm 1$ and $\pm 2$ on a small number of knots, and these remaining examples can be checked by comparing hyperbolic invariants. These results make use of the surgery formula for Heegaard Floer homology, which has already proved to be a powerful tool for obstructing cosmetic surgeries; we get stronger obstructions than previously known by considering the full graded theory. We make use of a new graphical interpretation of knot Floer homology and the surgery formula in terms of immersed curves, which makes the grading information we need easier to access.
\end{abstract}

\section{Introduction}
Given a knot $K$ in $S^3$, two surgeries $S^3_r(K)$ and $S^3_{r'}(K)$  with $r\neq r'$ are said to be \emph{cosmetic} if $S_r(K)$ and $S_{r'}(K)$ are diffeomorphic as unoriented manifolds, and \emph{truly cosmetic} if $S_r(K) \cong S_{r'}(K)$ (here, and throughout the paper, $\cong$ denotes orientation preserving diffeomorphism). Surgeries that are cosmetic but not truly cosmetic are called \emph{chirally cosmetic}. Cosmetic surgeries are one way in which the surgery characterization of a 3-manifold can fail to be unique. Examples of chirally cosmetic surgeries are not difficult to find, but Gordon conjectured that there are no truly cosmetic surgeries on nontrivial knots \cite[Conjecture 6.1]{Gordon}  (see also \cite[Problem 1.81 A]{KirbyProblemList}). This conjecture is stated more generally for knots in arbitrary 3-manifolds, with the notion of truly cosmetic surgery suitably extended, but we will only consider the case of knots in $S^3$.

\begin{conjecture}[Cosmetic Surgery Conjecture in $S^3$]\label{conj:cosmetic}
Let $K$ be a nontrivial knot in $S^3$. If $r \neq r'$, then $S^3_r(K) \not\cong S^3_{r'}(K)$.
\end{conjecture}

The conjecture can be viewed as a generalization of the knot complement problem, solved by Gordon and Luecke \cite{GordonLuecke}, which states that no pair of cosmetic surgeries contains the trivial surgery $S^3_\infty(K)$. In addition to this, several partial results related to Conjecture \ref{conj:cosmetic} are known. Boyer and Lines used surgery formulae for Casson-Walker and Casson-Gordon invariants to place a restriction on $\Delta_K(t)$ for knots $K$ admitting truly cosmetic surgeries \cite{BoyerLines}. Much of the recent progress on Conjecture \ref{conj:cosmetic} has made use of Heegaard Floer homology, which has lead to several results obstructing cosmetic or truly cosmetic surgeries. For any pair of truly cosmetic surgeries, the surgery slopes were shown first to have opposite signs \cite{OzSz:rational-surgeries, Wu}, and then to in fact be opposite \cite{NiWu}. If $K$ admits truly cosmetic surgeries, then the genus of $K$ is not one \cite{Wang} and the knot Floer homology of $K$ satisfies certain additional constraints \cite{NiWu, Gainullin}.

Heegaard Floer homology has already proved to be a very powerful tool at distinguishing surgeries, but it has not been used to its full potential. Each application to the cosmetic surgery conjecture mentioned above uses only partial information from Heegaard Floer homology, either the total rank of $\HFhat$ \cite{OzSz:rational-surgeries, Wu}, the $d$-invariants \cite{Wang, NiWu}, or the Euler characteristic of $\HFplus_{red}$ \cite{NiWu, Gainullin}. We will harness (almost) all of the information in Heegaard Floer homology to obtain much stronger obstructions to truly cosmetic surgeries. In particular, we consider the isomorphism type of $\HFhat$ as an absolutely graded vector space, which amounts to keeping track of the grading for each generator in addition to the rank. This is facilitated by a recent reinterpretation of Heegaard Floer invariants for manifolds with torus boundary in terms of collections of immersed curves due to the author, Rasmussen, and Watson \cite{HRW, HRW:companion}. In particular, this provides a combinatorial framework that makes comparing gradings for surgeries on knots easier. We give obstructions to truly cosmetic surgeries in terms of numerical invariants $n_s$ extracted from knot Floer homology, the Heegaard Floer thickness $th(K)$, and the Seifert genus $g(K)$. The Heegaard Floer thickness is the difference between the maximal and minimal $\delta$-grading in $\HFKhat$. The integers $n_s$ will be introduced in Section \ref{sec:invts-from-gamma}; briefly, they count the intersection number of the immersed multicurve representing knot Floer homology with a horizontal line at height $i$. Our main result is the following:
\begin{theorem}\label{thm:main}
If $K\subset S^3$ is a nontrivial knot and $S^3_r(K) \cong S^3_{r'}(K)$ for $r \neq r'$, then
\begin{itemize}
\item[(i)] The pair of slopes $\{r,r'\}$ is either $\{\pm 2\}$ or $\{\pm \tfrac 1 q\}$ where $\displaystyle q =  \frac{n_0 + 2 \sum_{s = 1}^\infty n_s}{4 \sum_{s=1}^\infty s^2 n_s}$;
\item[(ii)] if $\{r,r'\}$ is $\{\pm 2\}$ then $g(K) = 2$ and $n_0 = 2 n_1$;
\item[(iii)] if $\{r,r'\}$ is $\{\pm \tfrac 1 q\}$ then $\displaystyle q \le \frac{th(K) + 2g(K)}{2g(K)(g(K)-1)}$.
\end{itemize}
\end{theorem}

Note that for any given knot we rule out all but at most two pairs of truly cosmetic surgeries; it was not previously known that a knot must have finitely many pairs of truly cosmetic surgery. Importantly, this is an \emph{effective} finiteness statement, meaning that for a given knot $K$ the two (or fewer) potential truly cosmetic surgery pairs are explicitly determined. This makes it possible to check the conjecture for a given finite set of knots by means of a finite computation (see Theorem \ref{thm:16-crossings-connected-sums} below). In fact, in the vast majority of cases observed, the value of $q$ predicted by conclusion $(i)$ is not an integer and the conditions in $(ii)$ are not met, so that Theorem \ref{thm:main} rules out all truly cosmetic surgeries on $K$.

Although the formula in conclusion $(i)$ of Theorem \ref{thm:main} determines $q$, the bound on $q$ in conclusion $(iii)$ is also useful as a convenient way to rule out truly cosmetic surgeries on large classes of knots without computing the $n_s$ invariants. In particular:

\begin{corollary}\label{cor:small-thickness}
If $K$ is a nontrivial knot in $S^3$ with $g(K) > 2$ and $th(K) < 6$, then $K$ does not admit any truly cosmetic surgeries.
\end{corollary}
\begin{proof}
The cosmetic surgery conjecture is known for genus one knots \cite{Wang}, so the assumption that $g(K) > 2$ rules out slopes $\pm 2$. The bound from Theorem \ref{thm:main} then implies that $q < 1$, which is impossible.
\end{proof}

A great many knots have $th(K) < 6$. A knot is \emph{Floer homologically thin} if $th(K) = 0$, that is if only one $\delta$-grading is occupied (we will simply refer to such knots as \emph{thin}); examples of thin knots include all alternating and quasialternating knots. Furthermore, direct computation reveals that for any prime knot $K$ with at most 16 crossings, $th(K) \le 2$. It follows that the cosmetic surgery conjecture holds for any of these knots with genus other than two.

An additional consequence of Theorem \ref{thm:main} is a restriction on the manifolds that could arise from truly cosmetic surgeries on knots in $S^3$:

\begin{corollary}\label{cor:homology1or2}
If $Y$ is a closed oriented 3-manifold with $|H_1(Y; \Z)| > 2$, then $Y$ can not be obtained by a truly cosmetic surgery on any knot in $S^3$.
\end{corollary}

In addition to those in Theorem \ref{thm:main}, some further conditions must also be satisfied by the knot Floer homology of $K$ for truly cosmetic surgeries to exist. These conditions require some more definitions to state in general (see Section \ref{sec:explicit-obstructions}), but in the case of thin knots they can be stated in terms of the Alexander polynomial and signature.

\begin{theorem}\label{thm:alternating}
If a nontrivial knot $K \subset S^3$ is thin (in particular if $K$ is alternating or quasialternating) and admits a pair of truly cosmetic surgeries, then $\Delta_K(t) = nt^2 - 4nt + (6n+1) - 4nt^{-1} + nt^{-2}$ for some positive integer $n$, $\sigma(K) = 0$, and the surgery slopes are $\pm 1$ or $\pm 2$.
\end{theorem}

We remark that Theorem \ref{thm:alternating} is the best statement possible for thin knots using only Heegaard Floer homology. That is, if $K$ is thin, $\Delta_K(t) = nt^2 - 4nt + (6n+1) - 4nt^{-1} + nt^{-2}$, and $\sigma(K) = 0$, then the pairs $\{S^3_1(K), S^3_{-1}(K)\}$ and $\{S^3_2(K), S^3_{-2}(K)\}$ are not distinguished by their Heegaard Floer homology. The first two examples in the knot tables where this occurs are the knots $9_{41}$ and $9_{44}$. Similarly, Theorem \ref{thm:main} and the additional constraints in Section \ref{sec:explicit-obstructions} allow us to extract as much information as possible from Heegaard Floer homology: if we cannot prove the cosmetic surgery conjecture for a given knot, then the pairs of surgeries that are not ruled out in fact have isomorphic Heegaard Floer homology. 

Knots having any surgeries at all that cannot be distinguished by Heegaard Floer homology are exceedingly rare, but they do exist---to date we have found 337 such knots, each with two pairs of slopes that are not distinguished. All of these 337 examples are genus two and have Alexander polynomial of the form described in Theorem \ref{thm:alternating} (though not all are thin). To prove the conjecture for these knots, we must use other invariants to distinguish the remaining pairs of surgeries; for the examples found so far the hyperbolic volume and Chern-Simons invariant are sufficient. In this way we verify the conjecture for all prime knots up to 16 crossings, and in fact for arbitrary connected sums of such knots.

\begin{theorem}\label{thm:16-crossings-connected-sums} Let $K \subset S^3 $ be a nontrivial knot whose prime summands each have at most 16 crossings. If $r \neq r'$ then $S^3_r(K) \not\cong S^3_{r'}(K)$. 
\end{theorem}

This paper grew out of an attempt to answer the question: how much can Heegaard Floer homology tell us about the Cosmetic Surgery Conjecture? For knots in $S^3$, we have now given a comprehensive answer to that question. Indeed, we see that Heegaard Floer homology can say a great deal, and examples for which it is not sufficient to prohibit truly cosmetic surgeries appear to be very rare. Nevertheless, other tools will be required to prove Conjecture \ref{conj:cosmetic} outright. 

The results in this paper are, to the author's knowledge, the strongest obstructions available for knots in $S^3$, but similar results can be obtained using other techniques. In particular, very recently Futer, Purcell, and Schleimer have used hyperbolic methods to prove a result comparable to Theorem \ref{thm:main}: for any given hyperbolic knot, they rule out truly cosmetic surgeries on all but an explicitly determined finite set of slopes \cite{FuterPurcellSchleimer}. For knots in $S^3$ this result seems to be weaker in practice than Theorem \ref{thm:main}, in the sense that the finite set of slopes remaining is larger and thus the exhaustive search required to check the conjecture on a given set of knots is slower. However, the result in \cite{FuterPurcellSchleimer} applies not just to hyperbolic knot complements in $S^3$ but to arbitrary finite volume hyperbolic manifolds with torus boundary. In contrast, the arguments in this paper are highly specialized to knots in $S^3$ (though $S^3$ can be replaced with any integer homology sphere L-space). It is interesting to ask how much Heegaard Floer homology can tell us about cosmetic surgeries in more general manifolds. Although the proofs in this paper are not well suited to that setting, the broader principle of using immersed curves to more easily compare the ranks and relative gradings of the Heegaard Floer invariants of different Dehn fillings may be fruitful. In particular, we could hope to obtain a finiteness result in the line of Theorem \ref{thm:main} and \cite[Theorem 7.29]{FuterPurcellSchleimer}.

\begin{question}
Can Heegaard Floer homology be used to rule out all but finitely many pairs of cosmetic surgery slopes for arbitrary manifolds with torus boundary?
\end{question}

A good starting point would be graph manifolds with torus boundary, since (i) hyperbolic techniques would not apply and (ii) there is well developed machinery for understanding the bordered Floer invariants in this case \cite{Hanselman:graph, HW}.

The rest of the paper is organized as follows. In Section \ref{sec:knot-floer} we describe knot Floer homology and review its relevant properties. This section is recommended even for readers already familiar with knot Floer homology, as our description of the invariant is not the usual one. In particular we describe knot Floer homology as a decorated collections of immersed curves, a perspective that we will use throughout the paper. Section \ref{sec:survey} addresses the Cosmetic Surgery Conjecture and briefly reviews some existing results; this is not meant to be a comprehensive survey of the subject, but rather focuses on results that use Heegaard Floer homology and on which our arguments build. In Section \ref{sec:new-obstructions} we introduce our main obstructions and prove Theorem \ref{thm:main}. Section \ref{sec:explicit-obstructions} refines these results and provides several explicit obstructions to a knot admitting truly cosmetic surgeries; in particular, we prove Theorem \ref{thm:alternating}. Finally, in Section \ref{sec:computation} we verify the conjecture for arbitrary connected sums of knots up to 16 crossings, proving Theorem \ref{thm:16-crossings-connected-sums}.

{\bf Acknowledgements: } I am grateful to Dave Futer for sparking my interest in this problem and for helpful correspondence, and to Liam Watson for comments on an earlier draft of this paper. 

\section{Knot Floer homology}\label{sec:knot-floer}

Knot Floer homology was defined by Oszv{\'a}th and Szab{\'o} \cite{OzSz:knot-floer} and independently by J. Rasmussen \cite{Ras:knot-floer}. We will use a description of this invariant for knots in $S^3$ in terms of immersed curves; this is rather different from the original formulation, though it carries equivalent information. We will primarily be interested in a weaker form of the invariant, which we call $\hatGamma(K)$ and which is equivalent to the $UV = 0$ truncation of the knot Floer complex. The $UV = 0$ truncation of knot Floer homology is also equivalent to bordered Floer homology of the knot complement, and an immersed curve description of this invariant is due to the author, Rasmussen, and Watson \cite{HRW, HRW:companion} (the case of knot complements is discussed specifically in \cite[Section 4]{HRW:companion}). In particular, the invariant denoted $\hatGamma(K)$ in this paper agrees with $\HFhat(M)$ with $M = S^3 \setminus \nu(K)$ in the notation of \cite{HRW, HRW:companion}. For readers unfamiliar with bordered Floer homology, a bordered free construction of the immersed curves $\hatGamma(K)$ will appear in a forthcoming paper by the author \cite{Hanselman:CFK}. This construction has the advantage that it can be strengthened to a decorated curve $\Gamma(K)$ capturing the full knot Floer complex $\CFKinfty(K)$. We will not need this stronger invariant in the present paper, though we will need to make use of the construction in \cite{Hanselman:CFK} in one small way (see Proposition \ref{prop:d-invariant}).

We will now describe the invariant $\hatGamma(K)$. Throughout we work with coefficients in $\F = \Ztwo$. We begin by setting notation for the spaces in which the curves $\hatGamma(K)$ appear. Let $\torus$ denote the torus marked with a chosen pair of parametrizing curves $\mu$ and $\lambda$ and a single marked point $w$, which we may take to be the intersection of $\mu$ and $\lambda$. Let $\cylinder$ denote the infinite cyclic covering space of $\torus$ in which $\lambda$ lifts to a loop and $\mu$ does not, and let $\bar{p}:\cylinder\to\torus$ denote the covering map. We will identify $\cylinder$ with $(\R / \Z) \times \R$, where the lifts of $\lambda$ and $\mu$ are horizontal and vertical, respectively, and the preimages of $w$ are the points $(0,s-1/2)$ for integers $s$.  Let $\plane$ denote the universal covering space $\R^2$ with covering map $\widetilde{p}:\plane\to\cylinder$. By slight abuse, we will often refer to the vertical line through the marked points in $\cylinder$ (or through a column of marked points in $\plane$) as $\mu$, though it is really a lift of the curve $\mu$ in $\torus$. Finally, we will use $\ptorus$, $\pcylinder$, and $\pplane$ to denote corresponding punctured surfaces obtained by removing the marked points. We may conflate punctures and marked points at times, the only distinction is that we use marked points if we want to allow disks to cover these points and we use punctures otherwise.

\subsection{ The knot Floer invariant.}

To a knot $K$ in $S^3$, we associate a collection $\curveset = \{\gamma_0, \ldots, \gamma_n\}$ of oriented immersed curves in  $\pcylinder$; $\curveset$ is an invariant knot $K$, up to regular homotopy and reindexing of the curves (note that we work in the punctured cylinder $\pcylinder$ rather than marked cylinder $\cylinder$, meaning that the curves and homotopies are required to avoid the punctures). $\hatGamma(K)$ denotes this multicurve along with two extra decorations: (1) each curve may be decorated with a local system, and (2) the multi-curve carries decorations to encode relative grading data. We remark that the first decoration is not relevant to the arguments in this paper, though we describe it here for completeness. Moreover, it is still unknown whether the local system decoration is nontrivial for any knot in $S^3$. In contrast, the second type of decoration will play a crucial role. This Maslov grading decoration for immersed curves in the context of bordered Floer homology is discussed in detail in \cite[Section 2]{HRW:companion}; see also \cite{Hanselman:CFK} for the special case of $\hatGamma(K)$.

\noindent \emph{Local systems}: If a curve $\gamma_i$ is homologous in $\pcylinder$ to $k_i$ copies of some primitive curve $\gamma_i'$, then we will assume that $\gamma_i$ is realized by $k_i$ parallel copies of $\gamma_i'$ outside of a small region, in which the curve crosses itself $k_i-1$ times as shown in Figure \ref{fig:local-system}(a). This region contains one segment of $\gamma_i$ (the negatively sloped segment in the figure) that intersects each of $k_i-1$ other segments. Then each curve $\gamma_i$ is decorated with a subset of these $k_i-1$ self intersection points. Note that for a primitive curve this decoration is automatically trivial. The selected intersection points should be interpreted as places where a traveler along the negatively sloped segment is allowed to make a left turn onto one of the other segments. In the language of \cite{HRW}, this means that we extend the curve $\gamma_i$ to an immersed train track and add two (oriented) edges near the selected intersection point, as shown in Figure \ref{fig:local-system}(b). Decorating a chosen subsection of intersection points by adding train track edges in this way determines a $k_i \times k_i$ invertible matrix with coefficients in $\F$ that counts immersed paths from the left to the right of the boxed region. By a local system we mean a similarity class of such matrices, which is equivalent to the subset of intersection points above since the matrix constructed in this way will be in rational canonical form. The local system associated to each curve in $\curveset$ is also an invariant of $K$. Note that the pair of train track edges added at an intersection point is equivalent to a single ``crossover arrow", in the shorthand notation of \cite{HRW}, and using the arrow sliding moves the whole configuration can be replaced with $k_i$ parallel copies of $\gamma_i'$ with some crossover arrows between parallel strands. This also defines a matrix, which is similar to the one constructed above.

\begin{figure}
\labellist
  \pinlabel {$\gamma_0$} at 125 22
  \pinlabel {$(a)$} at 60 -10
         \endlabellist
\includegraphics[scale=1]{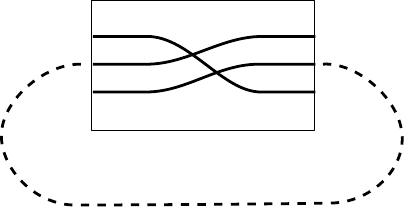} \hspace{3 cm}
\raisebox{5 mm}{\labellist
  \pinlabel {$=$} at 62 15
  \pinlabel {$(b)$} at 62 -22
         \endlabellist\includegraphics[scale=1]{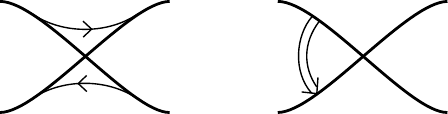}}

\vspace{5 mm}

\caption{We assume that non-primative curves are arranged as shown in $(a)$ and decorated with a subset of the intersection points in the boxed region. To each chosen intersection point we add a pair of edges as pictured in $(b)$, or equivalently a crossover arrow in the notation of \cite{HRW}. Counting (smooth) paths from the left side of the boxed region to the right determines a matrix, which can be interpreted as a local system.}
\label{fig:local-system}
\end{figure}

\noindent \emph{Maslov Gradings}: The multicurve $\curveset$ can be enhanced with extra decorations in order to encode some (relative) grading information. In fact, the desired information is already contained in the immersed curve for any single component $\gamma_i$ of $\curveset$, so extra decorations are only required to capture relative gradings between different components. There are multiple ways to encode this information, the approach we describe here is to add labeled arcs to $\curveset$ connecting different components. More precisely, we extend the multicurve $\curveset$ to an immersed graph $\curveset_{\gr}$, which contains $\curveset$ as a subgraph and all of whose vertices are contained in $\curveset$, but which also contains some number of edges connecting vertices on different curve components. We will refer to these new edges as \emph{grading edges}, and they should be ignored expect for the purposes of computing gradings. We require that the grading edges are tangent to $\curveset$ at their endpoints, so that $\curveset_{\gr}$ is in fact an immersed train track (recall that a train track is a graph for which all incident edges at any vertex are mutually tangent). Moreover, we require the ends of the grading edges to be consistent with the relative orientations on the curves, in the sense that a smooth path that runs over an edge connecting $\gamma_i$ and $\gamma_j$ either follows the orientation on both curves or opposes the orientation on both curves. Grading edges themselves are directed (this direction is not required to agree with the orientation on the curves) and labeled with an integer weight.

We say that a set of grading edges on $\curveset$ is \emph{complete} if $\curveset_{\gr}$ is connected as a graph. We say that a set of grading edges is \emph{consistent} if, for any closed (not necessarily smooth) path $P$ in $\curveset_{\gr}$, 
\begin{equation}\label{eq:maslov-consistency}
- \text{rotation}(P) + \text{winding}(P) + \text{weights}(P) = 0,
\end{equation}
where rotation$(P)$ is $\tfrac{1}{2\pi}$ times the total counterclockwise rotation along the smooth sections of $P$, winding$(P)$ is the sum over marked points $w$ in $\cylinder$ of the winding number of $P$ around $w$, and weights$(P)$ is the sum of the weights of all grading edges traversed by $P$, with the weight counted negatively if $P$ traverses the grading edge backwards. More precisely, for the rotation and winding numbers to make sense, we only consider paths $P$ that do not wrap around the cylinder; these can be viewed as paths in the marked strip obtained by cutting $\cylinder$ open along the line $\{\tfrac 1 2\}\times\R$. We say that two complete sets of grading edges are \emph{equivalent} if their union is consistent. With these definitions established, the grading decoration we will use on the multicurve $\curveset$ to define $\hatGamma(K)$ is a complete consistent set of grading edges; this decoration is an invariant of $K$ up to equivalence of sets of grading edges.

\parpic[r]{
 \begin{minipage}{50mm}
 \centering
 \includegraphics[scale=1]{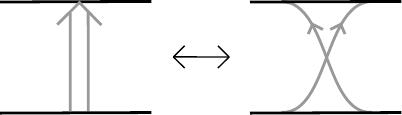}
 \captionof{figure}{A crossover arrow is shorthand for a pair of edges.}
 \label{fig:crossover-convention}
  \end{minipage}
  }
For any complete set of grading edges and any additional oriented grading edge, there is a unique integer weight on the new edge that makes the combined set consistent. In particular, for any oriented edge connecting two components of $\curveset$, tangent to $\curveset$ at each end in a way that is consistent with the orientation of the curves, we can choose a set of grading edges that contains this edge with some integer weight. If $\curveset$ contains $n+1$ curves, then a minimal complete subset of any complete set of grading edges contains $n$ edges. We will generally choose a minimal complete set of grading edges, except that the convention established in \cite{HRW:companion} replaces each grading edge with a pair of edges as in Figure \ref{fig:crossover-convention}, the shorthand for which is a bold arrow. We will call these pairs of grading edges \emph{grading arrows}, and we will decorate $\curveset$ with a minimal consistent set of grading arrows. Using both edges in a grading arrow is unnecessary but can be convenient. Note that a grading arrow can be labeled by a single weight, since consistency requires both edges in the pair to have the same weight. For examples of multicurves decorated with grading arrows, see Figure \ref{fig:curve-examples}.

\begin{remark} Grading arrows are described in detail in \cite[Section 2]{HRW:companion} (see also \cite{Hanselman:CFK} for a slightly different description). Note that in \cite{HRW:companion}, grading arrows do not carry an integer weight; these arrows should be interpreted as having weight zero. By sliding arrows over punctures and changing the corresponding weights by one, it is clear that any configuration of weighted arrows can be replaced with an equivalent configuration of arrows that all have weight zero. Thus using weighted grading arrows is not necessary, but it is convenient as it provides greater freedom in which arrows we choose. Finally note that the arrow weights discussed here should not be confused with the complexity weights on crossover arrows used in the proof of the arrow removal algorithm in \cite[Section 3.7]{HRW}
\end{remark}

\subsection{Properties of $\hatGamma$} \label{sec:properties} Some examples of the invariant $\hatGamma(K)$ are shown in Figure \ref{fig:curve-examples}. For the unknot and the right hand trefoil, the invariant contains a single curve. The invariant for the figure eight knot consists of two curves decorated with a single grading arrow connecting them, while the invariant for $9_{44}$ has five curves and four grading arrows. These examples demonstrate some general properties of $\hatGamma$, which we now highlight.

(I) For a knot $K$ in $S^3$, the multicurve $\curveset$ associated with $\hatGamma$ can be homotoped to have exactly one intersection with the line $\{\tfrac 1 2\}\times\R$ in $\cylinder \simeq (\R/\Z)\times \R$. This follows from the fact that when $\hat\Gamma$ and this vertical line are in minimal position, their intersections generate $\HFhat$ of the meridional filling, that is, of $\HFhat(S^3) \cong \F$ \cite[Theorem 2]{HRW}. This condition implies that there is one distinguished curve component, which we always take to be $\gamma_0$, that wraps around the cylinder once, and that all other components can be contained in a neighborhood of $\mu$, the vertical line through the marked points. As in the examples in Figure \ref{fig:curve-examples}, $\gamma_0$ is always oriented left-to-right.

(II) The immersed curves in $\curveset$ are \emph{unobstructed}, meaning that they do not bound any ``teardrops," or one-cusped disks, that do not enclose a marked point. That is, there are no immersed disks in the punctured cylinder $\pcylinder$ with boundary on some $\gamma_i \in \curveset$ such that the boundary is a smooth path apart from one acute corner at a self intersection point of $\gamma_i$.

(III) The consistency condition for sets of grading arrows in \eqref{eq:maslov-consistency} is stated as a condition that must hold for all closed loops in the train track $\curveset_{\gr}$ that do not wrap around the puncture. We remark that this condition must in particular hold for each curve component $\gamma_i$ with $i \neq 0$, even before grading arrows are introduced, and this places restrictions on the allowed curves: for any closed curve with net zero rotation, the total winding number around punctures must also be zero. In particular, any figure eight shaped curve must enclose the same number of marked points on each side.

\begin{remark}\label{rmk:only-figure-eights} In the examples above, all of the curves $\gamma_i$ with $i\neq 0$ are figure eights wrapping around two adjacent punctures; such a curve will be called a \emph{simple figure eight}. This is not a general property of $\hatGamma$, but it is incredibly common. In fact, this condition holds for all but one prime knot up to 15 crossings. The unique exception is $15n166130$, for which $\hatGamma$ contains (along with thirty simple figure eights) two components not of this form. These components are still figure eight curves enclosing one marked point on each side, but they enclose non-adjacent marked points. Larger examples can be constructed with figure eight curves enclosing more than one marked point on each side, but the author has not yet found an example with a homologically trivial curve that is not a figure eight in this more general sense.
\end{remark}

(IV) The decorated curve set $\hatGamma(K)$ is invariant under rotation by $\pi$ about the origin, up to homotopy of curves and equivalence of grading arrows, except that the rotation flips the orientation of every curve. This is the geometric expression of a symmetry for bordered Floer invariants established in \cite[Theorem 7]{HRW:companion}, which was proved earlier in the case of knot complements by Xiu \cite{Xiu}.

\begin{figure}
\includegraphics[scale=1]{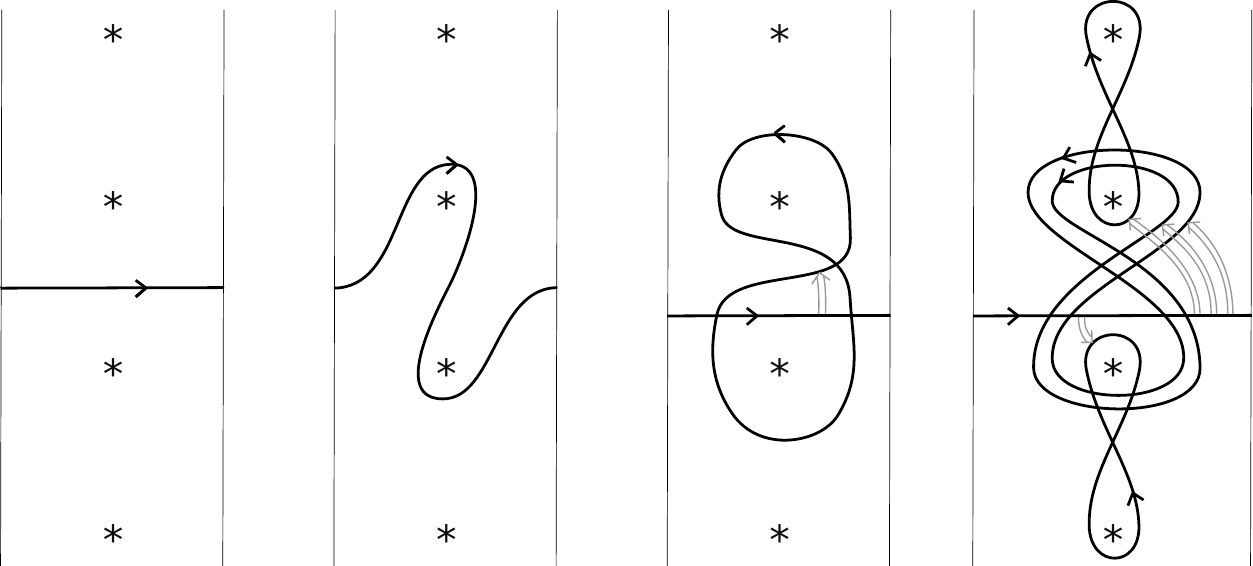}
\caption{The curve invariants $\hatGamma(K)$ for, from left to right, the unknot, the right-handed trefoil, the figure eight knot, and $9_{44}$. Grading arrows, shown in gray, all have weight 0, and there are no local systems.}
\label{fig:curve-examples}
\end{figure}

\subsection{Invariants derived from knot Floer homology}\label{sec:invts-from-gamma}
 
Several interesting numerical invariants of $K$ can be extracted from $\hatGamma$. For example, the genus of $K$ is the maximum height of an intersection of $\hatGamma$ with the vertical line $\mu$ through the marked points, assuming $\hatGamma$ is in minimal position with $\mu$. Here we mean height in the discrete sense: an intersection point is said to occur at height $s$ if its $y$-coordinate falls between the marked points at $(0,s-\tfrac 1 2)$ and $(0,s+\tfrac 1 2)$.

There is a distinguished curve component $\gamma_0$ that wraps around the cylinder $\cylinder$ and a distinguished intersection of $\gamma_0$ with $\mu$, the first time $\gamma_0$ reaches $\mu$ after wrapping around the cylinder. The Ozsv{\'a}th-Szab{\'o} $\tau$ invariant is the height of this first intersection point on $\gamma_0$ (see \cite[Section 4.2]{HRW:companion}; this intersection corresponds to the generator of vertical homology and the height gives its Alexander grading). Moreover, after $\gamma_0$ reaches this first intersection point, it can do one of three things: turn right (downward), turn left (upward), or continue straight. Hom's invariant $\epsilon(K)$ is $1$, $-1$, or $0$, respectively, in these three cases (again see \cite[Section 4.2]{HRW:companion}; this behavior corresponds to the generator of vertical homology lying at the end of a horizontal arrow, the beginning of a horizontal arrow, or no horizontal arrow). Note that by symmetry $\gamma_0$ can only continue straight if the intersection was at height 0, so $\epsilon(K) = 0$ implies $\tau(K) = 0$. In this case $\gamma_0$ intersects $\mu$ only once and is homotopic to the simple horizontal line $S^1 \times \{0\}$ in the cylinder. $\tau$ and $\epsilon$ are both concordance invariants of $K$; in fact, it can be shown that the distinguished curve component $\gamma_0$ is itself a concordance invariant of $K$, up to homotopy of curves (this follows from \cite[Theorem 1]{Hom:survey}; $\gamma_0$ is the immersed curve representative of the direct summand of $UV = 0$ knot Floer homology supporting the homology, which is a concordance invariant). In fact the curve $\gamma_0$ exactly encodes the $\epsilon$-equivalence class as defined by Hom. Note that $\tau$ and $\epsilon$ depend only on $\gamma_0$.

\parpic[r]{
 \begin{minipage}{30mm}
 \centering
 \includegraphics[scale=1.4]{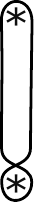}
 \captionof{figure}{A simple figure eight curve contributes two vertical segments.}
 \label{fig:pulled-tight-fig8}
 \vspace{2 mm}
  \end{minipage}
  }
  In the arguments in this paper, it will be useful to quantify a few more aspects of the underlying multicurve $\curveset$ for $\hatGamma$. To do this, we will assume that $\curveset$ has a convenient form. We start by ``pulling tight" as described in \cite{HRW}. This means that we assume $\curveset$ is the minimal length representative of its homotopy class, subject to the constraint that it avoids an open disk of some small radius $\epsilon$ around each marked point. Intuitively, we think of there being a peg of radius $\epsilon$ at each marked point and think of the curve as a rope winding through the pegs; we pull the curve taut as if the rope were elastic\footnote{To ensure transverse self intersection (and make pictures easier to read), we can modify this by letting the curve wrap at a slightly different radius each time it encounters a peg}. Under this assumption, $\curveset$ breaks into segments that connect pegs (separated by small portions of $\curveset$ lying on the boundary of a peg). Because of property (I) from the previous section, exactly one of these segments leaves a neighborhood of $\mu$ and wraps around the cylinder, while all other segments connect a peg to one directly above it and thus are (roughly) vertical of length one. We are interested in counting these vertical segments. We say that a vertical segment is at height $s$ if it connects a peg at height $s - \tfrac 1 2$ to a peg at height $s + \tfrac 1 2$; let $n_s$ denote the number of vertical segments at height $s$ and let $n = \sum_{s\in\Z} n_s$ denote the total number of vertical segments. In general we will perturb the curve slightly from its pulled tight position so that the segments counted by the $n_s$ are in fact vertical outside of a neighborhood of the pegs; Figure \ref{fig:pulled-tight-fig8} shows a simple figure eight curve in this position. Note that a simple figure eight curve always contributes two vertical segments at the same height. The unique non-vertical segment in $\curveset$ is also of interest, and we will record the slope $m$ of this segment. For example, consider the curve invariant for the knot $9_{44}$ shown in Figure \ref{fig:curve-examples}. Each figure eight curve contributes two vertical segments, and we have $m = 0$, $n_0 = 4$, $n_1 = n_{-1} = 2$, and $n_s = 0$ for all other $s$. Since they are derived from $\hatGamma$, the quantities $m$ and $n_s$ are invariants of the knot, though we remark that they are not fundamentally new and can be described in terms of other invariants (for example, it can be shown that $m = 2\tau(K) - \epsilon(K)$).

\subsection{$\hatGamma(K)$ and bifiltered complexes.}\label{sec:bifiltered-complexes}

Readers who are already familiar with knot Floer homology will notice that the object defined above bears little resemblance to the original formulation of the invariant, which takes the form of a $\Z$-graded, $\Z\oplus\Z$-filtered chain complex $\CFKminus(K)$ defined up to filtered chain homotopy equivalence. To reassure these readers that the two invariants are in fact equivalent, we pause to briefly describe how the knot Floer complex can be recovered from $\hatGamma(K)$. More accurately, we recover the so-called $UV = 0$ quotient of this complex, which records only the horizontal and vertical differentials in $\CFKminus(K)$; the stronger invariant $\Gamma(K)$ described in \cite{Hanselman:CFK}, which is $\hatGamma(K)$ equipped with some extra decoration, would be required recover the full knot Floer complex.

Let $\hatGamma = \hatGamma(K)$ be represented by the immersed multicurve $\curveset$ along with decorations as described above, and let $\mu$ denote the vertical line through the marked points in $\cylinder$. We construct a complex $C_{\hatGamma}$ over $\F[U,V]/(UV=0)$ whose generators are the intersection points of $\curveset$ with $\mu$, and whose differential counts immersed bigons for which the left boundary lies on $\mu$ and the right boundary is a path in $\curveset$. More precisely, for intersection points $x$ and $y$ in $\curveset \cap \mu$, a bigon from $x$ to $y$ is a homotopy classes of maps $f: D^2 \to \cylinder$ such that $f(-i) = x$, $f(i) = y)$, the negative real part of $\partial D^2$ maps to $\mu$, the positive real part of $\partial D^2$ maps to $\curveset$, $f$ is an immersion away from $i, -i$, and $f(\partial D^2)$ forms acute corners at $x$ and $y$. Let $N(x,y)$ denote the mod 2 count of such bigons. We are interested in recording how these bigons cover certain marked points. The marked points of $\cylinder$ all lie on the line $\mu$; we will push each of these points $w$ slightly off of $\mu$ to the right, and add a new marked point $z$ next to each just to the left of $\mu$. For any homotopy class $\phi$ of maps as above, we define $n_z(\phi)$ and $n_w(\phi)$ to be the multiplicity with which a representative of $\phi$ covers the $z$'s and $w$'s, respectively. The differential then is given by
$$\partial(x) = \sum_y N(x,y) U^{n_w(x,y)} V^{n_z(x,y)} y.$$
Since we set $UV = 0$ in our coefficient ring, the differential only needs to count bigons that cover either $w$ marked points or $z$ marked points, but not both. We note that to recover the full knot Floer complex, we would need to count bigons covering both types of marked points and  we would not set $UV = 0$. However, if we attempt to construct such a complex using only $\hatGamma$, $\partial^2$ may not be zero. To correctly recover the knot Floer complex, we need to take into account some extra decorations in the stronger invariant $\Gamma(K)$ (see \cite{Hanselman:CFK}).

We set an Alexander grading on the generators of $C_{\hatGamma}$, which are intersection points between $\curveset$ and $\mu$, by their height: for $x \in \curveset\cap\mu$, we define $A(x) \in \Z$ to be $s$ if $x$ lies between the marked points at $(0,s-\tfrac 1 2)$ and $(0,s+\tfrac 1 2)$. It is clear that a bigon from $x$ to $y$ covers $k$ marked points of type $w$, then $A(y) = A(x) + k$, and if it covers $k$ marked points of type $z$ then $A(y) = A(x) - k$; thus if $\partial(x)$ contains a term $U^a V^b y$ where one of $a$ or $b$ vanishes, then $A(y) = A(x) - b + a$.

 In addition to the Alexander grading, $C_{\hatGamma}$ carries an integer Maslov grading $M$. This satisfies
 \begin{equation}\label{eq:maslov-shift}
 M(Ux) = M(x) -2, \qquad M(Vx) = M(x), \quad \text{and} \quad M(\partial x) = M(x) -1.
 \end{equation}
These relationships determine $M$ as a relative grading on each connected component of $C_{\hatGamma}$, since if $U^a V^b y$ appears in $\partial x$ then $M(x) - M(y) = 1 - 2a$. The connected components of $C_{\hatGamma}$ correspond directly to the component immersed curves in $\hatGamma$. $M$ can be extended to a relative grading on all of $C_{\hatGamma}$ by considering bigons between $\mu$ and the train track $\curveset_{\gr}$ obtained by including grading arrows with $\curveset$; we require that \eqref{eq:maslov-shift} still holds for these bigons, where running over a grading edge of weight $k$ forward (resp. backward) counts as covering both $U$ and $V$ $k$ times (resp. $-k$ times). That is, if there is a bigon from $x$ to $y$ whose left boundary lies in $\mu$ and whose right boundary is a smooth path in $\curveset_{\gr}$ that covers $w$ marked points $n_w$ times and $z$ marked points $n_z$ times, and for which $k$ is the sum of the weights (counted with sign) of all grading edges traversed on the boundary of the bigon traveling from $x$ to $y$, then 
\begin{equation}\label{eq:maslov-shift2}
M(y) - M(x) = -1 + 2n_w + 2k.
\end{equation}
We can always assume that all grading arrows in $\hatGamma$ lie completely to the right of $\mu$ or completely to the left of $\mu$, so to determine the relative Maslov grading it is sufficient to consider bigons that cover only $w$'s or only $z$'s and that include at most one grading arrow. That said, \eqref{eq:maslov-shift2} applies for bigons covering both types of marked points, and can in fact be generalized to the following formula for the grading difference between any two generators:
\begin{definition}\label{def:maslov-grading-difference}
For $x, y \in \curveset\cap\mu$, let $P_1$ be a path (not necessarily smooth) from $x$ to $y$ in $\curveset_{\gr}$, let $P_2$ be a path from $y$ to $x$ in $\mu$, and let $P$ be the concatenated path $P_1 P_2$. $P$ is a closed path that is smooth apart from right corners at $x$ and $y$ and possibly one or more cusps. Let $\text{rotation}(P)$ denote $\tfrac{1}{2\pi}$ times the total counterclockwise rotation along the smooth sections of $P$, let $\text{winding}_w(P)$ denote the net winding number of $P$ around $w$ marked points, and let $\text{weights}(P)$ be the sum of weights (counted with sign) of all grading edges traversed by $P$. Then
\begin{equation}\label{eq:maslov-difference}
M(y) - M(x) = - 2\cdot \text{rotation}(P) +  2\cdot\text{winding}_w(P) + 2\cdot \text{weights}(P).
\end{equation}
\end{definition}

Note that the mod 2 reduction of the (relative) Maslov grading is determined by the sign of the intersection points in $\curveset\cap\mu$. It is clear that the completeness condition on sets of grading arrows ensures that this relative grading is defined for all generators, and the consistency condition ensures that the relative grading is well defined.  Finally, this relative grading can be promoted to an absolute grading by noting that there is a special generator of $C_{\hatGamma}$, the first intersection of $\gamma_0$ with $\mu$ after $\gamma_0$ wraps around the cylinder; we set the Maslov grading of this generator to be 0.

As originally defined, the knot Floer complex $\CFKinfty(K)$ is a $\Z$-graded, $\Z\oplus\Z$-filtered chain complex finitely generated over $\F[U,U^{-1}]$. The first filtration is given by negative powers of $U$, while the second, the Alexander filtration, is recorded separately. The Alexander filtration is determined by an Alexander grading on the generators, together with the fact that multiplication by $U$ lowers the filtration level by one. It is convenient to add a second formal variable $V$ to keep track of the Alexander filtration, giving rise to a complex generated (with the same generating set) over $\F[U^{\pm 1}, V^{\pm 1}]$, with the two filtrations given by negative powers of $U$ and $V$, respectively. The original definition is then recovered from this by setting $V = 1$, though the new complex is bigger as many powers of $V$ can be attached to the same element of $\CFKinfty(K)$. To get a complex over $\F[U^{\pm 1}, V^{\pm 1}]$ that is isomorphic to $\CFKinfty(K)$, we would consider only elements $U^a V^b x$ such that $a - b = A(x)$; then a generator $x$ of $\CFKinfty(K)$ would correspond to the element $V^{-A(x)} x$, which is at Alexander filtration level $A(x)$. Terms in the differential that fix the algebraic filtration level (resp. the Alexander filtration level) are referred to as vertical arrows (resp. horizontal arrows); setting $UV = 0$ amounts to counting only horizontal and vertical arrows. $C_{\hatGamma}$ as defined above recovers this $UV = 0$ quotient complex.

\parpic[r]{
 \begin{minipage}{30mm}
 \centering
 \labellist
  \pinlabel {$a$} at 29 76
  \pinlabel {$b$} at 23 48
  \pinlabel {$c$} at 23 37
  \pinlabel {$d$} at 23 26
  \pinlabel {$e$} at 28 4
  \pinlabel {$z$} at 20 63
  \pinlabel {$w$} at 31 63
  \pinlabel {$z$} at 20 15
  \pinlabel {$w$} at 31 15
         \endlabellist
 \includegraphics[scale=1.4]{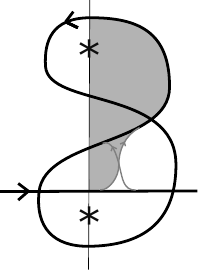}
 \captionof{figure}{}
 \label{fig:fig8-grading}
 \vspace{2 mm}
  \end{minipage}
  }
\begin{example}
Consider the figure eight knot, whose invariant $\hatGamma$ is shown in Figure \ref{fig:curve-examples}. To recover the knot Floer complex, we would draw the vertical line $\mu$ through the marked points, placing basepoints $z$ and $w$ to the left and right of each marked point, and notice that there are 5 intersections with $\mu$; label these $a$, $b$, $c$, $d$, $e$ from top to bottom as in Figure \ref{fig:fig8-grading}. The Alexander grading is 1 for $a$, 0 for $b$, $c$, and $d$, and $-1$ for $e$. There are two bigons to the left of $\mu$ contributing $Vb$ to $\partial a$ and $Ve$ to $\partial c$, and there are two bigons to the right of $\mu$ contributing $Ua$ to $\partial c$ and $Ub$ to $\partial e$. The distinguished generator with Maslov grading 0 is $d$. When we include the grading arrow, there is a bigon on the right of $\mu$ from $d$ to $a$ that covers the marked point $w$ once and whose boundary runs over the grading arrow (see Figure \ref{fig:fig8-grading}; note that for clarity we have drawn the pair train track edges that the grading crossover arrow represents, the boundary runs over one of these). The crossover arrow has has weight 0, so this bigon implies that $M(a) = 1$. The bigons mentioned previously imply that $M(b)=M(c) = 0$ and $M(e)=-1$.
\end{example}

\begin{remark}
The discussion above shows that it is fairly straight forward to construct a $UV = 0$ bifiltered complex from a decorated set of immersed curves. The converse, that any $UV = 0$ bifiltered complex can be represented by a decorated set of immersed curves and that this representation is unique in an appropriate sense, is more difficult. This follows from the main Theorem in \cite{HRW}, which proves a related result for type D structures, since the $UV = 0$ quotient of $\CFKinfty(K)$ is equivalent to the bordered Floer invariant of the knot complement $\CFD(S^3 \setminus \nu(K))$.  See also \cite{Hanselman:CFK} for a proof that does not pass through bordered Floer homology. $\hatGamma(K)$ is defined to be the decorated immersed curve that represents the $UV = 0$ quotient of $\CFKinfty(K)$.
\end{remark}

\subsection{Surgery formula}

A key strength of the knot Floer homology package is that there is a simple way to recover the Heegaard Floer homology of any Dehn surgery on a knot $K$. In particular, $\HFhat(S^3_{p/q}(K))$ can be realized as the intersection Floer homology of the decorated immersed curve $\hatGamma(K)$ with lines of slope $p/q$ in the punctured torus $\ptorus$ or the punctured cylinder $\pcylinder$, as we will now further explain (a precise statement is given in Theorem \ref{thm:surgery-formula}).

More precisely, let $\bar p(\hatGamma)$ be the projection of $\hatGamma(K)$ to $\ptorus$; we will see that $\HFhat(S^3_{p/q}(K))$ agrees with intersection Floer homology of $\bar p(\hatGamma)$ with a straight line $\ell_{p,q}$ of slope $p/q$. By this we mean the homology of a chain complex $CF(\bar p(\hatGamma),\ell_{p,q})$ generated by intersection points whose differential counts immersed bigons with right boundary on $\bar p(\hatGamma)$ and left boundary on $\ell_{p,q}$. We do not allow bigons to cover the marked point (we indicate this by taking Floer homology in the punctured torus $\ptorus$ rather than the marked torus $\torus$). We count bigons whose boundary includes crossover arrows associated with the local system decoration on $\hatGamma$ (see Figure \ref{fig:local-system}(b)), though it turns out that including these bigons in the differential has no effect on the resulting homology, so in practice the local systems on $\hatGamma$ can be ignored. In contrast, we do not count bigons whose boundary runs over a grading arrow so the Maslov decoration has no effect on the differential, but it will be used to define gradings on the resulting complex.

There are two types of grading information on intersection Floer homology. First, $CF(\bar p(\hatGamma), \ell_{p,q})$ decomposes into spin$^c$ summands, where generators $x$ and $y$ are in the same summand if and only if the loop $P$ formed by concatenating a path from $y$ to $x$ in $\ell_{p,q}$ followed with a (not necessarily smooth) path from $x$ to $y$ in $\bar p(\curveset_{\gr})$ (that is, in $\bar p(\hatGamma)$ with grading arrows included) is nullhomologous. This decomposition is easier to understand by lifting to the covering space $\pcylinder$, where we take Floer homology of $\hatGamma$ with lifts of $\ell_{p,q}$; to recover the same complex we must use multiple different lifts of $\ell_{p,q}$, and the spin$^c$ summands are precisely the Floer homology of $\hatGamma$ with any one lift of $\ell_{p,q}$. On each spin$^c$ summand there is also a (relative) Maslov grading, where the grading difference $M(x) - M(y)$ is defined exactly as in Definition \ref{def:maslov-grading-difference}. The general form of the grading difference can be cumbersome, but in practice it is sufficient to consider bigons that involve at most one grading arrow, possibly with a cusp at one end of the grading arrow.

\begin{lemma}\label{lem:maslov-easy-bigons}
Suppose $x$ to $y$ are connected by an immersed region bounded between $\curveset_{\gr}$ and a lift of $\ell_{p,q}$ that is 
\begin{itemize}
\item[$(a)$] a bigon not involving a grading arrow, 
\item[$(b)$] a bigon whose $\curveset_{\gr}$ boundary is a smooth path traversing one grading arrow of weight $m$, or 
\item[$(c)$] a cusped bigon whose $\curveset_{\gr}$ boundary traverses one grading arrow of weight $m$ with a single left turning cusp,
\end{itemize}
as pictured in Figure \ref{fig:grading-bigons}. Suppose in any case that the region covers $k$ marked points (counted with multiplicity). Then the Maslov grading difference $M(y) - M(x)$ is given by $-1 + 2k$ in case $(a)$, $-1+2k+2m$ in case $(b)$, or $2k + 2m$ in case $(c)$
\end{lemma} 
\begin{proof}
This follows from the general formula for $M(y) - M(x)$; see \eqref{eq:maslov-difference}. In cases $(a)$ and $(b)$, the net counterclockwise rotation traversing the closed loop from $x$ to $y$ in $\curveset_{\gr}$ and from $y$ to $x$ in $\ell_{p,q}$ is $2\pi$, but since this includes two right angles at $x$ and $y$ the net rotation along the smooth segments is $\pi$; it follows that the term $-2\text{rotation}(P)$ in \eqref{eq:maslov-difference} is $-1$. In case $(c)$ the extra cusp means that the net rotation along smooth segments in the boundary is 0. In each case the net winding number around the marked points is $k$, and in cases $(b)$ and $(c)$ the grading arrow contributes $m$ to weights$(P)$.
\end{proof}

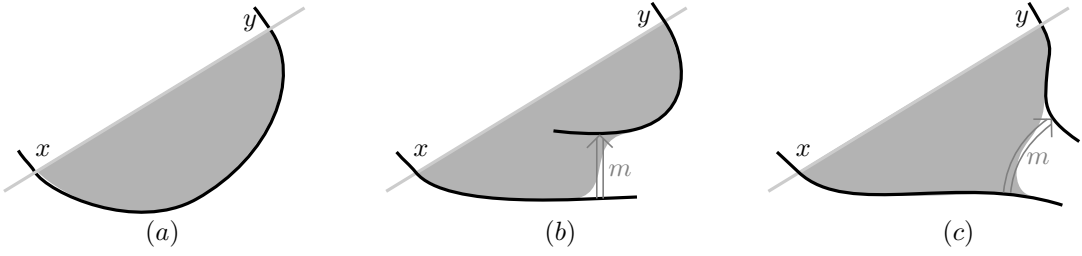
\begin{figure}
\labellist
  \pinlabel {$x$} at 10 16
  \pinlabel {$y$} at 62 48
  \pinlabel {$x$} at 105 16
  \pinlabel {$y$} at 161 49
  \pinlabel {$x$} at 201 16
  \pinlabel {$y$} at 256 49
  
  \pinlabel {$\color{gray}m$} at 155 11
  \pinlabel {$\color{gray}m$} at 260 14
  
  \pinlabel {$(a)$} at 40 -5
  \pinlabel {$(b)$} at 140 -5
  \pinlabel {$(c)$} at 240 -5
         \endlabellist
\includegraphics[scale=1.5]{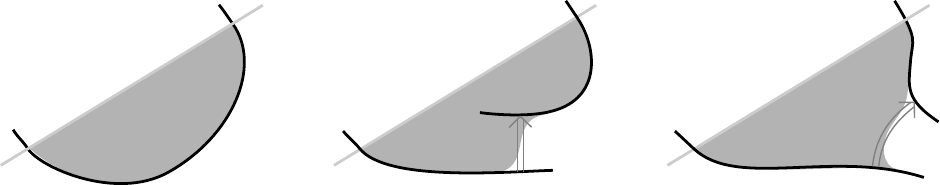}
\vspace{2 mm}

\caption{Three types of regions used to determine the relative Maslov grading on intersection Floer homology. Recall that a crossover arrow consists of a pair of edges in the train track as in Figure \ref{fig:crossover-convention}; the boundaries of the bigons in (b) and (c) use one of these edges.}
\label{fig:grading-bigons}
\end{figure}

The following theorem relates $\HFhat$ of surgery on a knot with intersection Floer homology as defined above. Recall that for $p/q$ surgery on a knot $K \subset S^3$ there is a canonical identification of $\spinc(S^3_{p/q}(K))$ with $\Zp$.

\begin{theorem}(Surgery formula)\label{thm:surgery-formula}
Consider a knot $K \in S^3$ and $p,q$ relatively prime with $p>0$. Fix a small $\epsilon > 0$ (in particular, $\epsilon < \frac 1 q$), and for each $i \in \Zp$ let $\ell_{p,q}^i$ be a straight line in the punctured cylinder $\pcylinder$ of slope $p/q$ that passes through the point $(0, -\tfrac 1 2 + \tfrac i q + \epsilon)$. Then
$$\HFhat( S^3_{p/q}, i) \cong HF( \hatGamma(K), \ell_{p,q}^i )$$
as relatively graded vector spaces, where the right hand side refers to intersection Floer homology in $\pcylinder$.
\end{theorem}

\begin{proof}
This is a special case of a much more general gluing formula for bordered Floer invariants (\cite[Theorem 2]{HRW} without gradings and \cite[Theorem 5]{HRW:companion} with gradings). Indeed, $\widehat\Gamma(K)$ is precisely the invariant $\HFhat(M)$ associated with the the knot complement $M = S^3 \setminus \nu(K)$,  the invariant $\HFhat(D^2\times S^1)$ is simply the meridian $\partial D^2 \times S^1$, and $p/q$-Dehn surgery corresponds to gluing $D^2 \times S^1$ to $M$ by a map taking the meridian to a line of slope $p/q$.
\end{proof}

A direct consequence of this is that $\rk( \HFhat( S^3_{p/q}(K) ) )$ is given by the minimal intersection in $\ptorus$ of $\bar p(\hatGamma)$ and $\ell_{p,q}$. This is because all bigons not covering a puncture can be removed by pulling $\bar p(\hatGamma)$ tight (here we need that no component of $\bar p(\hatGamma)$ is parallel to $\ell_{p,q}$ to ensure admissibility, but this is clear if $p \neq 0$). Similarly, $\rk( \HFhat( S^3_{p/q}(K), i ) )$ is the minimal intersection in $\pcylinder$ of $\hatGamma$ and $\ell_{p,q}^i$. We remark that when $\hatGamma$ is pulled tight as described in Section \ref{sec:invts-from-gamma}, it automatically has minimal intersection with each $\ell^i_{p,q}$.

Recall that if $\hatGamma$ is pulled tight, then outside of a neighborhood of the punctures it consists of a collection of some number $n$ of length one vertical segments and a single non-vertical segment of slope $m$. Then we have the following expression for $\rk( \HFhat( S^3_{p/q}(K) ) )$ (compare \cite[Proposition 9.5]{OzSz:rational-surgeries}):

\begin{proposition}\label{prop:total-rank}
With the integers $m$ and $n$ defined as above,  $\rk( \HFhat( S^3_{p/q}(K) ) ) = |p - mq| + n|q|$.
\end{proposition}
\begin{proof}
A line of slope $p/q$ in $\torus$ intersects the vertical line through the marked point $|q|$ times, so there are $n|q|$ intersection points coming from vertical segments in $\bar p(\hatGamma)$. The remaining intersection points come from intersections with the segment of slope $m$, and the number of such intersections is the distance between the slopes $p/q$ and $m/1$, namely $\left| \det \left( \begin{smallmatrix} p & q \\ m & 1 \end{smallmatrix} \right) \right|$.
\end{proof}

\begin{remark}\label{rmk:alternate-proof}
The key idea in the proof of \cite[Theorem 2]{HRW} (in the special case required for Theorem \ref{thm:surgery-formula}) is to perturb both $\hatGamma(K)$ and $\ell_{p,q}^i$ into a special form so that the intersection Floer chain complex can be directly identified with (one spin$^c$ component of) the box tensor product of two bordered Floer invariants, the homology of which is known to agree with $\HFhat(S^3_{p/q}, i)$. It is also possible to prove Theorem \ref{thm:surgery-formula} without passing through bordered Floer homology. This is accomplished by perturbing $\ell_{p,q}^i$ so that the intersection Floer chain complex $CF(\hatGamma(K), \ell_{p,q}^i)$ is identified with the complex $\widehat{\mathbb{X}}_i$ in the mapping cone formula  \cite{OzSz:rational-surgeries}, whose homology is also known to give $\HFhat(S^3_{p/q}, i)$. This identification was shown for large integer surgery in \cite[Proposition 70]{HRW:companion}, and the full proof will appear in \cite{Hanselman:CFK}. 
\end{remark}

While the two proofs are similar in spirit, this second proof has a few advantages since the mapping cone formula carries some information not available with bordered Floer homology. For example, the mapping cone formula recovers the absolute grading on $\HFhat(S^3_{p/q}, i)$ while bordered Floer homology can only give the relative grading. In addition, the identification of the Floer chain complex with the mapping cone formula can be generalized to one involving $\mathbb{X}^+_i$ instead of $\widehat{\mathbb{X}}_i$, so a version of Theorem \ref{thm:surgery-formula} holds for $+$ type invariants \cite{Hanselman:CFK}. We will not need the absolute grading or $+$ type invariants in the present paper, but we will make use of the identification mentioned in Remark \ref{rmk:alternate-proof} in one small way in Proposition \ref{prop:d-invariant} below; namely, we will use the fact that the subset of generators of $HF(\hatGamma(K), \ell_{p,q}^i)$ arising from any one curve component $\gamma_i$ in $\hatGamma(K)$ can be identified with the subset of generators of $S^3_{p,q}(K)$ arising from the mapping cone formula applied to direct summand of $\CFKinfty(K)$ corresponding to $\gamma_i$.

We will be interested in a special class of knots for which the bifiltered complex $\CFKinfty(K)$ has a direct summand that looks like $\CFKinfty(U)$, where $U$ is the unknot. That is, we require that for some choice of basis $\CFKinfty(K)$ has a generator with no differentials in or out. In this case we will say that $\CFKinfty(K)$ \emph{has an isolated generator}. Note that $\CFKinfty(K)$ having an isolated generator implies that the curve $\gamma_0$ in $\hatGamma(K)$ is homotopic to the horizontal curve wrapping around the cylinder once, but the latter condition is slightly weaker since $\gamma_0$ does not see diagonal arrows in $\CFKinfty(K)$ (giving an immersed curve condition equivalent to having an isolated generator would require the stronger invariant $\Gamma(K)$ and amounts to further imposing that $\gamma_0$ is not connected to any other $\gamma_i$ by the additional decorations in $\Gamma(K)$).

Recall that for a 3-manifold $Y$ with spin$^c$ structure $\spin$, the $d$-invariant, or correction term, $d(Y,\spin)$ is defined as the minimum absolute grading of an element of the image of $HF^\infty(Y, \spin)$ in $\HFplus(Y,\spin)$. Understanding $d$-invariants usually requires working the $+$ flavor of invariants; there is always a generator in $\HFhat(Y, \spin)$ whose absolute grading is the $d$-invariant, but without knowing the $U$-module structure on $\HFplus$ we generally have no way of knowing which generator gives the $d$-invariant (unless, of course, there is only one generator in $\HFhat(Y, \spin)$). However, if $Y$ is (nonzero) surgery on a knot $K$ for which $\CFKinfty(K)$ has an isolated generator, there is an obvious choice for a distinguished generator in each spin$^c$ structure and indeed this generator gives the $d$-invariant.

\begin{proposition}\label{prop:d-invariant}
Suppose $K$ is a knot for which $\CFKinfty(K)$ has an isolated generator, and in particular the distinguished curve $\gamma_0$ of $\hatGamma(K)$ is horizontal. Then for each $i \in \Zp$, the absolute grading of the generator of $\HFhat(S^3_{p/q}(K), i) \cong HF( \hatGamma(K), \ell_{p,q}^i )$ corresponding to the unique intersection point of $\gamma_0$ with $\ell_{p,q}^i$ is $d( S^3_{p/q}(K), i)$.
\end{proposition}
\begin{proof}
In the mapping cone formula, the direct summands of $\CFKinfty(K)$ give rise to direct summands for the mapping cone $\mathbb{X}_i$, and it is clear that to compute the $d$-invariant it is sufficient to consider only the unique non-acyclic summand of $\CFKinfty(K)$ and the corresponding summand of the mapping cone. When $\CFKinfty(K)$ has an isolated generator, the homology of this summand has rank one, so the $d$-invariant must be the grading of its only generator. We now appeal not just to Theorem \ref{thm:surgery-formula} but also to the identification of $CF(\hatGamma(K), \ell_{p,q}^i)$ with $\widehat{\mathbb{X}}_i$ mentioned in Remark \ref{rmk:alternate-proof}. The direct summands of $\CFKinfty(K)$ correspond to the curve components of $\hatGamma(K)$, with the non-acyclic summand corresponding to $\gamma_0$, and so the relevant summand of $\widehat{\mathbb{X}}_i$ is identified with the intersection Floer complex of $\gamma_0$ with $\ell_{p,q}^i$. Thus the grading of the unique generator of $HF(\gamma_0, \ell_{p,q}^i)$ is $d( S^3_{p/q}(K), i)$.
\end{proof}
This result is not at all surprising, but it does require the mapping cone formula proof of Theorem \ref{thm:surgery-formula} since the bordered Floer approach gives no way of confirming that the obvious distinguished summand $HF(\gamma_0, \ell_{p,q}^i)$ should capture the $d$-invariant. We remark that this use of Remark \ref{rmk:alternate-proof} in the proof of Proposition \ref{prop:d-invariant} is the only essential dependence of the present paper on \cite{Hanselman:CFK}.

\section{Obstructing truly cosmetic surgeries}\label{sec:survey}

We now turn to a brief survey of some past results on which the arguments in the next section build. The first observation is that, since $H_1(S^3_{p,q}(K)) \cong \Zp$, any pair of cosmetic surgery slopes must have the same numerator. The next constraint is a condition on the Alexander polynomial of $K$ proved by Boyer and Lines:

\begin{theorem}\cite[Proposition 5.1]{BoyerLines}\label{thm:Boyer-Lines}
If $K$ admits a truly cosmetic surgery, then $\Delta_K''(1) = 0$.
\end{theorem}
This result is a consequence of surgery formulas for the Casson-Walker invariant $\lambda$ and the Casson-Gordon invariant $\boldtau$:
\begin{eqnarray}
\lambda(S^3_{p/q}(K)) &=& \lambda(L(p,q)) + \frac{q}{2p} \Delta_K''(1) \label{eqn:CassonWalker}\\
\boldtau(S^3_{p/q}(K)) &=& \boldtau(L(p,q)) - \sigma(K, p), \label{eqn:CassonGordon}
\end{eqnarray}
where $\sigma(K,p) = \sum_{r=0}^{p-1} \sigma_K(e^{2i\pi r/p})$ does not depend on $q$. If $p/q$ and $p/ q'$ are truly cosmetic surgery slopes, \eqref{eqn:CassonGordon} implies that $\boldtau(L(p,q)) = \boldtau(L(p,q'))$. For a lens space, $\boldtau(L(p,q))$ is a constant multiple of $p\lambda(L(p,q))$, so in fact $\lambda(L(p,q)) = \lambda(L(p,q'))$. Then \eqref{eqn:CassonWalker} implies that either $q = q'$ or $\Delta_K''(1) = 0$.

Heegaard Floer homology entered the story when Ozsv{\'a}th and Szab{\'o} constructed a surgery formula in terms of knot Floer homology \cite{OzSz:rational-surgeries} and used it to prove the following proposition. As a demonstration of the machinery that will be used in this paper, we present a proof that is essentially equivalent to the one in \cite{OzSz:rational-surgeries} but is reframed in the language of the immersed curve surgery formula.
\begin{proposition}\cite[Theorem 1.5]{OzSz:rational-surgeries}
\label{prop:opposite-sign}
Suppose $S^3_{p/q_1}(K) \cong \pm S^3_{p/q_2}(K)$ with $q_1 \neq q_2$. Either $q_1$ and $q_2$ have opposite signs or $S^3_{p/q_1}(K)$ is an L-space.
\end{proposition}
\begin{proof}
We must have that $\rk( \HFhat( S^3_{p/q_1}(K) ) ) = \rk( \HFhat( S^3_{p/q_2}(K) ) )$. By Proposition \ref{prop:total-rank},
\begin{equation}\label{eq:equal-hat-ranks}
|p - mq_1| + n|q_1| = |p - mq_2| + n|q_2|,
\end{equation}
where $m$ is the slope of the non-vertical segment in $\hatGamma(K)$ and $n$ is the number of vertical segments. By taking the mirror of $K$ if necessary, we may assume without loss of generality that $m \ge 0$. First suppose that $q_1$ and $q_2$ are both negative or that they are both positive and greater than $\tfrac p m$. In either case, \eqref{eq:equal-hat-ranks} simplifies to $(m+n)q_1 = (m+n)q_2$. Since $m+n > 0$ for a nontrivial knot, this implies that $q_1 = q_2$. Next suppose that $q_1$ and $q_2$ are both positive and smaller than $\tfrac p k$; in this case \eqref{eq:equal-hat-ranks} simplifies to $(n-m)q_1 = (n-m)q_2$. If $q_1 \neq q_2$, we must have $n = m$, which implies that $K$ is an L-space knot, and since $\tfrac{p}{q_i} > m = 2g(K)-1$, the result of either surgery is an $L$-space. Finally, suppose that $q_1$ and $q_2$ are both positive, with $mq_1 < p$ and $mq_2 > p$; \eqref{eq:equal-hat-ranks} becomes $p - mq_1 + nq_1 = mq_2 - p + nq_2$. This implies
$$n(q_2 - q_1) = 2p - m(q_1 + q_2) < 2mq_2 - m(q_1 + q_2) = m(q_2 - q_1).$$
This is a contradiction, since $n \ge m$.
\end{proof}

In the case of truly cosmetic surgery on a knot $K$ with Seifert genus equal to one, Ozsv{\'a}th and Szab{\'o} in fact showed that the surgery must be an $L$-space \cite[Theorem 1.4]{OzSz:rational-surgeries}. Wang ruled out this possibility, implying that the cosmetic surgery conjecture holds for all genus one knots \cite{Wang}; we will give a new proof of this fact in Section \ref{sec:new-obstructions} (see Corollary \ref{cor:genus1}). Wu later ruled out the possibility that truly cosmetic surgeries are $L$-spaces for arbitrary knots \cite{Wu} by observing that the restrictions on the Alexander polynomial of an $L$-space knot given in \cite{OzSz:lens-space-surgeries} imply that $\Delta_K''(1) \neq 0$ and applying Theorem \ref{thm:Boyer-Lines}. Thus truly cosmetic surgery slopes have opposite sign.

A significant advancement came in the following result of Ni and Wu:

\begin{theorem}\cite[Theorm 1.2]{NiWu}\label{thm:NiWu}
Suppose $S^3_{p/q} \cong S^3_{p/q'}$ with $q' \neq q$. Then
\begin{itemize}
\item[(i)] $\tau(K) = 0$, where $\tau$ is the Ozsv{\'a}th-Szab{\'o} concordance invariant;
\item[(ii)] $q' = -q$; and
\item[(iii)] $q^2 \equiv -1 \pmod p$.
\end{itemize}
\end{theorem}
The key ingredient here was a surgery formula for the $d$-invariants in Heegaard Floer homology \cite[Proposition 1.6]{NiWu}. A consequence of the surgery formula is that for $p/q > 0$, the $d$-invariants of $S^3_{p/q}(K)$ are less than or equal to the corresponding $d$-invariants of $S^3_{p/q}(U) = L(p,q)$, with equality holding for all spin$^c$ structures if and only if $V_0(K) = H_0(K) = 0$, where $V_0$ and $H_0$ are integer invariants related to certain maps in the rational surgery formula. For $p/q' < 0$, the same relationship holds with the inequality reversed. Let $d(Y)$ denote $\sum_{\spin\in\spinc(Y)} d(Y,\spin)$. For a lens space, $d(L(p,q))$ is a constant multiple of the Casson-Walker invariant $\lambda(L(p,q))$, and it was already noted that for a truly cosmetic surgery Equations \eqref{eqn:CassonWalker} and  \eqref{eqn:CassonGordon} imply that $\lambda(L(p,q)) = \lambda(L(p,q'))$. Thus
$$d(S^3_{p/q}) \le d(L(p,q)) = d(L(p,q')) \le d(S^3_{p/q'}).$$
For a truly cosmetic surgery equality must hold, so $V_0(K) = H_0(K) = 0$. This in particular implies (i), and then by Proposition \ref{prop:total-rank} $\rk (\HFhat( S^3_{p/q}(K) ))$ is a linear function of $|q|$, which implies $(ii)$. $(iii)$ follows from the fact that $d(L(p,q)) = d(L(p,-q)) = -d(L(p, q))$, and an explicit formula for $\lambda(L(p,q))$ showing that $\lambda(L(p,q)) = 0$ if and only if $q^2 \equiv -1 \pmod p$.

In fact, the first conclusion is slightly understated, since the proof really shows that $V_0(K) = H_0(K) = 0$ \cite[Theorem 2.5]{NiWu}, and this is strictly stronger than $\tau(K)$ being zero. Hom showed that when this condition holds then $\CFKinfty(K)$ has an isolated generator \cite[Proposition 3.11]{Hom:survey}. Recall that by this we mean for some choice of basis $\CFKinfty(K)$ has a single generator with no differentials in or out. Hom's paper also shows that the bifiltered chain complex $\CFKinfty(K)$, taken up to filtered chain homotopy equivalence and up to adding and removing acyclic summands, is a concordance invariant from which all known Heegaard Floer concordance invariants can be derived; having an isolated generator is equivalent to this concordance invariant being trivial. In the language of immersed curves, $\CFKinfty(K)$ having an isolated generator implies that $\gamma_0$ is the horizontal curve wrapping around the cylinder once; this in turn is equivalent to $\epsilon(K)$ being $0$ and implies $\tau(K) = 0$. To summarize, we have the following implications:
$$\CFKinfty(K) \text{ has isolated genarator} \underset{\centernot\impliedby}{\implies} \gamma_0 \text{ is horizontal } \iff \epsilon(K) = 0 \underset{\centernot\impliedby}{\implies} \tau(K) = 0. $$
Thus, Ni and Wu really proved the following:
\begin{theorem}\cite[Theorem 1.2, enhanced]{NiWu}\label{thm:NiWu-enhanced}
Suppose $S^3_{p/q} \cong S^3_{p/q'}$ with $q' \neq q$. Then
\begin{itemize}
\item[(i)] $\CFKinfty(K)$ has an isolated generator. In particular, $\epsilon(K) = \tau(K) = 0$.
\item[(ii)] $q' = -q$; and
\item[(iii)] $q^2 \equiv -1 \pmod p$.
\end{itemize}
\end{theorem}
It makes sense that the original theorem was stated in terms of $\tau$ only, as $\epsilon$ had not been defined at that time and the condition that $V_0 = H_0 = 0$ or that $\CFKinfty$ has an isolated generator makes for a more cumbersome statement. However, this means that some implications of Ni and Wu's work, which has already found many wonderful applications, have been overlooked. For example, the following result follows immediately from Theorem \ref{thm:NiWu-enhanced} and a cabling formula of Hom \cite[Theorem 2]{Hom:cables}, which says that $\epsilon$ of a cable is never zero:

\begin{corollary}\label{cor:genus-one}
The cosmetic surgery conjecture holds for any nontrivial cable of a knot in $S^3$.
\end{corollary}

This result was recently proved in \cite{Tao} using Theorem \ref{thm:NiWu} and Hom's cabling formula for $\tau$ \cite[Theorem 1]{Hom:cables} to rule out many cases, but other methods were needed to deal with cables for which $\tau = 0$.

We end this section with a technical result that will be required later, related to one used by Ni and Wu in the proof of Theorem \ref{thm:NiWu}. Recall that part $(iii)$ of that Theorem follows from the fact that if $p/q$ is a truly cosmetic surgery slope then the sum of all the $d$-invariants of $L(p,q)$ must be zero. More precisely, there is an explicit formula for this sum of $d$-invariants \cite[Lemmas 2.2 and 4.3]{Ras:lens-space-surgeries}:

\begin{equation}\label{eq:sum-of-dinvts}
d(L(p,q)) := \sum_{i = 0}^{p-1} d(L(p,q),i) = p \lambda(L(p,q)) = -\frac{1}{12}\left[ q + q' + p\sum_{i=1}^n (a_i - 3) \right],
\end{equation}

where $q'$ is the unique integer $0<q' <p$ with $qq' \equiv 1 \mod p$ and $[a_1, \ldots, a_n]$ is the Hirzebruch-Jung continued fraction expansion for $p/q$. If this sum is 0, then considering the term in the brackets modulo $p$ implies that $q \equiv -q' \mod p$ and Theorem \ref{thm:NiWu}$(iii)$ follows. We will at times be interested in only the first $q$ $d$-invariants of $L(p,q)$; below we show that when the sum of all $p$ $d$-invariants of $L(p,q)$ is zero, then the sum of the first $q$ of them is nonzero.

\begin{lemma}\label{lem:first-q}
For $p > q > 0$ relatively prime and $q^2 \equiv -1 \pmod p$, then $\displaystyle \sum_{i = 0}^{q-1} d(L(p,q), i) \neq 0$.
\end{lemma}
\begin{proof}
We will show that the sum is nonzero modulo $1/12$. We use the recursive formula for $d$-invariants of $L(p,q)$ given by Ozsv{\'a}th and Szab{\'o} \cite[Proposition 4.8]{OzSz:absolutely-graded}:
$$d(L(p,q),i) = -\frac 1 4 + \frac{(2i + 1-p - q)^2}{4pq} - d(L(q,r), j),$$
where $r$ and $j$ are the mod $q$ reductions of $p$ and $i$, respectively. In particular,
$$\sum_{i=0}^{q-1} d(L(p,q),i) = -\sum_{i=0}^{q-1} \frac 1 4 + \sum_{i=0}^{q-1}\frac{(2i + 1-p - q)^2}{4pq} - \sum_{i=0}^{q-1}d(L(q,r), i)$$
The third sum on the right hand side is simply $d(L(q,r))$, and it it is easy to see from \eqref{eq:sum-of-dinvts} that this is an integer multiple of $1/12$. The first sum on the right hand side, which evaluates to $q/4$, is also a multiple of $1/12$, so it is enough to check that the second sum is not.
$$\sum_{i=0}^{q-1}(2i + 1-p - q)^2 = \sum_{i=0}^{q-1}\left(p^2 + 2p(q - 1 - 2i) + (q-1-2i)^2 \right)
= qp^2 + 2p\sum_{i=0}^{q-1} (q - 1 - 2i) + \sum_{i=0}^{q-1} (q-1-2i)^2$$
The second term in the expression on the right is 0, since the summands run evenly from $q-1$ to $-(q-1)$. The final term is twice the sum of the first $\tfrac q 2$ odd squares if $q$ is even, or twice the sum of the first $\tfrac{q-1}{2}$ even squares if $q$ is odd; in either case, the sum evaluates to $\frac{q(q-1)(q+1)}{3}$. Thus we need to show that
$$\frac{qp^2 + q(q^2 -1)/3}{4pq} = \frac{p}{4} + \frac{q^2 - 1}{12p}$$ 
is not a multiple of $1/12$. The first term clearly is, but the second term is not as long $q^2 \not\equiv 1 \mod p$. This holds in particular when $q^2 \equiv -1 \pmod p$, unless $p=2$. We complete the proof by directly checking the case $p=2, q=1$: the claim holds since $d(L(2,1),0) = 1/4 \neq 0$.
\end{proof}

\section{New obstructions}\label{sec:new-obstructions}

Throughout this section we fix a knot $K$, and let $\hatGamma = \hatGamma(K)$ with underlying set of immersed curves $\curveset$. We will assume that $\CFKinfty(K)$ has an isolated vertex, which by Theorem \ref{thm:NiWu-enhanced} is necessary for $K$ to admit a truly cosmetic surgery. In particular this means that the distinguished curve $\gamma_0$ in $\curveset$ is horizontal. Theorem \ref{thm:NiWu-enhanced} also says that any pair of cosmetic surgery slopes are opposite, so we fix $p, q >0$ relatively prime, and set $Y_+ = S^3_{p/q}(K)$ and $Y_- = S^3_{-p/q}(K)$. Our goal is to obstruct $Y_+$ and $Y_-$ from being orientation preserving diffeomorphic by finding conditions under which $\HFhat(Y_+)$ and $\HFhat(Y_-)$ are not isomorphic as graded vector spaces.

The results in the previous section primarily make use of the Casson-Walker and Casson-Gordon invariants, the total rank of $\HFhat$, and the $d$-invariants, which can be viewed as the Maslov grading of one special generator of $\HFhat$ for each spin$^c$ structure. To extract more information and produce new obstructions, we will need to use the Maslov grading of all generators. In particular, the set of gradings of all generators of $\HFhat$ is an invariant, as is the partitioning of this set into subsets according to spin$^c$ structures. To avoid working with absolute gradings, we define $$\mrel(x) = M(x) - d(Y, \spin)$$ for $x$ in $\HFhat(Y, \spin)$. Since $\gamma_0$ is horizontal, for each $i \in \Zp \cong \spinc(Y_\pm)$ there is a distinguished generator in $\HFhat(Y_\pm, i)$ coming from the unique intersection point of $\gamma_0$ with $\ell_{p,\pm q}^i$; we will denote this generator $x_0^i$. By Proposition \ref{prop:d-invariant} the absolute grading of $x_0^i$ is $d(Y_\pm, i)$. Thus for $Y_+$ and $Y_-$, $\mrel(x)$ is simply the Maslov grading relative to the distinguished generator in the same spin$^c$ structure, i.e. $\mrel(x) = M(x) - M(x_0^i)$. We will consider the following multisets (that is, sets with repetition allowed) of relative gradings:
\begin{eqnarray*}
\mrel(Y) &=& \{ \mrel(x) | x \text{ a generator of } \HFhat(Y)\} \\
\mrel(Y,\spin) &=& \{ \mrel(x) | x \text{ a generator of } \HFhat(Y,\spin)\}
\end{eqnarray*}
These are invariants of $Y$ and the pair $(Y, \spin)$, respectively. In particular, if $Y_+ \cong Y_-$ then the sets $\mrel(Y_+)$ and $\mrel(Y_-)$ agree. Moreover, there is some permuation $\sigma$ on $\Zp$ such that $\mrel(Y_+, i) = \mrel(Y_-, \sigma(i))$. We will at times refer to the sum of all elements in these sets, which we denote $ \mrelsum(Y)$ and $\mrelsum(Y,\spin)$, respectively. 

\begin{remark}

Both $\spinc(Y_+)$ and $\spinc(Y_-)$ can be identified with $\Zp$ in a way that is canonical \emph{given the surgery description}, but this identification is not an invariant of the manifold. Thus even if $Y_+ \cong Y_-$, the $i$th spin$^c$ structure of $Y_+$ need not agree with the $i$th spin$^c$ structure of $Y_-$; this is why the permutation $\sigma$ is required above.\end{remark}

It is easy to see that the ranks of $\HFhat(Y_+)$ and $\HFhat(Y_-)$ agree. Indeed, since $\gamma_0$ is horizontal, the slope $m$ of the non-vertical segment in $\hatGamma$ is 0, so by Proposition \ref{prop:total-rank}
$$\rk(\HFhat(Y_+)) = p + n|q| = p + n|-q| = \rk(\HFhat(Y_-)).$$ Our main strategy for studying the sets of gradings described above is to define for each $i \in \Zp$ a particular map $\phi_i \co \HFhat(Y_+, i) \to \HFhat(Y_-, i)$ which is an isomorphism of ungraded vector spaces. In other words, $\phi_i$ gives a one-to-one correspondence between generators of $\HFhat(Y_+, i)$ and generators of $\HFhat(Y_-, i)$. This correspondence does not preserve the relative grading $\mrel$, even in the case that $\HFhat(Y_+)$ and $\HFhat(Y_-)$ are isomorphic as graded vector spaces; however we will be able to say explicitly how $\mrel$ changes under $\phi_i$ and we can use this to determine if the sets of gradings defined above are fixed. Combining these maps for all $i$ gives an (ungraded) isomorphism $\phi \co \HFhat(Y_+)\to\HFhat(Y_-)$.

We will assume that $\hatGamma$ has the form described in Section \ref{sec:invts-from-gamma} and shown in Figure \ref{fig:pulled-tight-fig8}. That is, we assume the curve is pulled tight, noting that outside of a neighborhood of the marked points each curve $\gamma_i$ with $i\neq 0$ consists of a collection of roughly vertical segments, and we perturb the curves slightly so that these are in fact parallel vertical segments (see Figure \ref{fig:reflection-map}). The endpoints of these vertical arcs are connected in some way within the neighborhoods of the marked points, but this information will not be relevant to us. Recall that $n_s$ denotes the number of these vertical segments at height $s$, and $n = \sum_{s = -\infty}^\infty n_s$ is the total number of vertical segments. When $\hatGamma$ is pulled tight in this way, it is clear that it intersects minimally with $\ell_{p,q}^i$ and $\ell_{p,-q}^i$, so we may view $\HFhat(Y_\pm, i) \cong HF(\hatGamma, \ell_{p,\pm q}^i)$ as generated by $\hatGamma \cap \ell_{p,\pm q}^i$. We will now define $\phi_i$ by describing a one-to-one correspondence between $\hatGamma \cap \ell_{p,q}^i$ and $\hatGamma \cap \ell_{p,-q}^i$. Note that $\ell_{p,-q}^i$ is the reflection across $\mu$ of $\ell_{p,q}^i$. In particular each intersection of $\ell_{p,q}^i$ with $\mu$ is also an intersection of $\ell_{p,-q}^i$ with $\mu$. The vertical segments are arbitrarily close to $\mu$, so each intersection of $\ell_{p,q}^i$ with a vertical segment can be uniquely identified with an intersection point in $\mu \cap \ell_{p,q}^i$, namely the nearest such intersection point. Conversely, for each point in $\mu \cap \ell_{p,q}^i$ and for each vertical segment at the same height as that point, there is exactly one nearby intersection of the vertical segment with $\ell_{p,q}^i$. The same is true for intersections between a vertical segment and $\ell_{p,-q}^i$. If $x \in \hatGamma \cap \ell_{p,q}^i$ lies on a vertical segment in $\hatGamma$, then we define $\phi_i(x)$ to be the point on $x \in \hatGamma \cap \ell_{p,-q}^i$ that lies on the same vertical segment in $\hatGamma$ and corresponds to the same point in $\mu \cap \ell_{p,q}^i = \mu \cap \ell_{p,-q}^i$; see the right side of Figure \ref{fig:reflection-map} depicting a neighborhood of one point in $\mu \cap \ell_{p,\pm q}^i $. If $x \in \hatGamma \cap \ell_{p,q}^i$ does not lie on a vertical segment in $\hatGamma$ then it is the unique point in $\gamma_0 \cap \ell_{p,q}^i$, corresponding to the distinguished generator of $\HFhat(Y_+, i)$; in this case we define $\phi_i(x)$ to be the unique point in $\gamma_0 \cap \ell_{p,-q}^i$, so $\phi_i$ identifies the distinguished generators.

\begin{figure}
\labellist
  \pinlabel {$\color{gray}\ell_{p,q}^i$} at 81 186
  \pinlabel {$\color{gray}\ell_{p,-q}^i$} at 22 186
  \pinlabel {$\gamma_0$} at 70 100
  
  \pinlabel {$\mu$} at 213 185
  \pinlabel {$\color{gray}\ell_{p,q}^i$} at 268 150
  \pinlabel {$\color{gray}\ell_{p,-q}^i$} at 270 45

\tiny
  \pinlabel {$a$} at 167 66
  \pinlabel {$b$} at 191 90
  \pinlabel {$c$} at 215 114
  \pinlabel {$d$} at 239 138
  
  \pinlabel {$\phi(a)$} at 163 134
  \pinlabel {$\phi(b)$} at 187 110
  \pinlabel {$\phi(c)$} at 228 90
  \pinlabel {$\phi(d)$} at 252 66  
         \endlabellist
\includegraphics[scale=1]{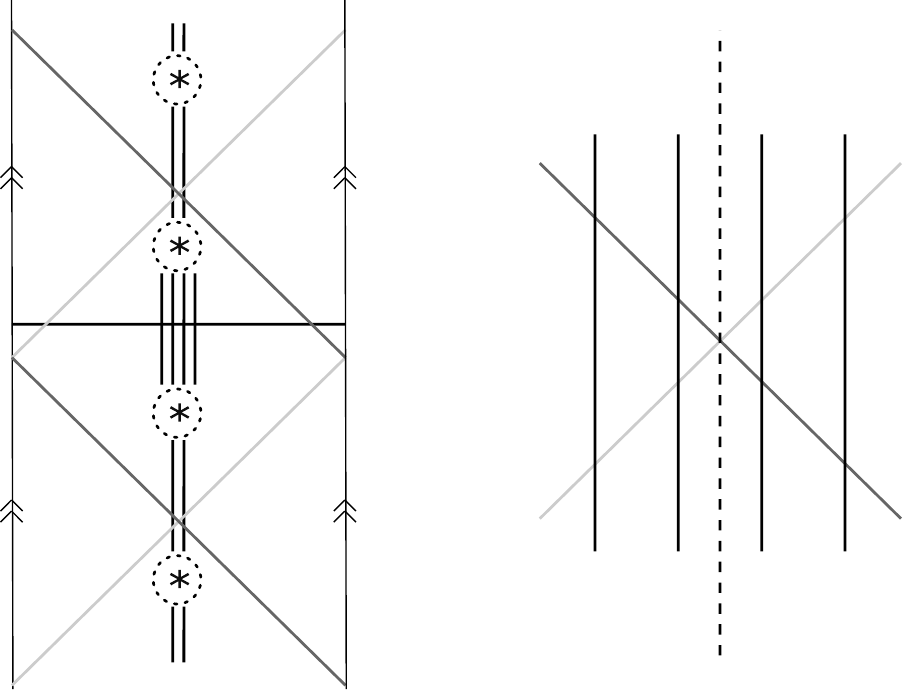}

\caption{The intersection of $\hatGamma$ (black) with a line $\ell_{p,q}^i$ (light gray) and it's vertical reflection $\ell_{p,-q}^i$ (dark gray). We assume that $\gamma_0$ is a simple horizontal curve, and that outside of a neighborhood of the marked points the remaining immersed curves in $\hatGamma$ consist of parallel vertical segments. The map $\phi_i$ takes an intersection of $\ell_{p,q}$ with a vertical segment to the nearby intersection of $\ell_{p,-q}^i$ with that segment, as pictured to the right.}
\label{fig:reflection-map}
\end{figure}
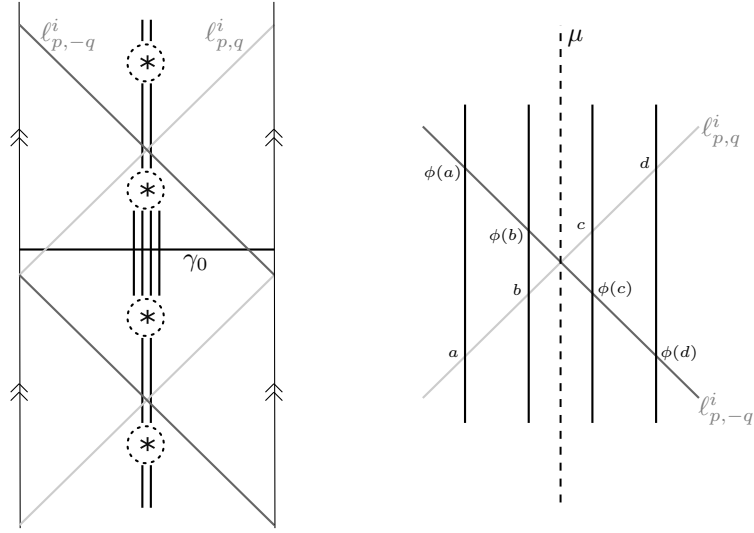

We often need to distinguish the generators in $\HFhat(Y_\pm)$ that come from intersections on vertical segments of $\hatGamma$ from the distinguished generators come from intersections on $\gamma_0$, which we denote by  $x_0^i$ in the $i$th spin$^c$ structure. Mimicking the standard notation for $\HFplus$, we will use $\HFhat_{red}(Y_\pm, i)$ to denote the summand of $\HFhat(Y_\pm, i)$ obtained by removing the generator $x_0^i$, and $\HFhat_{red}(Y_\pm)$ to denote $\bigoplus_{i\in\Zp} \HFhat_{red}(Y_\pm, i)$. Of course, reduced Floer homology in the hat setting does not make sense in general, but in this case where we know that the distinguished generator $x_0^i$ gives the $d$-invariant and corresponds to the tower in $\HFplus$, the analogy is appropriate.

For each generator $x$ of $\HFhat_{red}(Y_+)$, we are  are interested in computing both $\mrel(x)$ and $\mrel(\phi(x))$ (for the distinguished generators $x_0^i$, both quantities are 0 by definition). $x$ corresponds to an intersection point between $\ell_{p,q}^i$ for some $i$ and a vertical segment of $\hatGamma$. There are two integers we will associate with such an intersection point $x$. First, let $A(x)$ denote the height of the relevant vertical segment. Second, after a slight perturbation we can assume that the vertical segment containing $x$ lies exactly on the vertical line $\mu$ through the marked points of $\cylinder$ and that the intersection point $x$ can be viewed as an intersection point of $\mu$ and $\ell_{p,q}^i$; let $k(x)$ denote the number of marked points, counted with multiplicity, in the interior of the triangle formed by $\mu$, $\ell_{p,q}^i$, and $\gamma_0$. It is easiest to picture this triangle in the covering space $\plane$, as shown in Figure \ref{fig:triangle}. Using these quantities, we can compute the effect that the map $\phi$ has on the relative grading of $x$.

\begin{figure}
\labellist
  \pinlabel {$x$} at 152 135
  \pinlabel {$\gamma_0$} at 90 35
  \pinlabel {{\color{gray}$\ell_{p,q}^i$}} at 86 100
         \endlabellist
\includegraphics[scale=1]{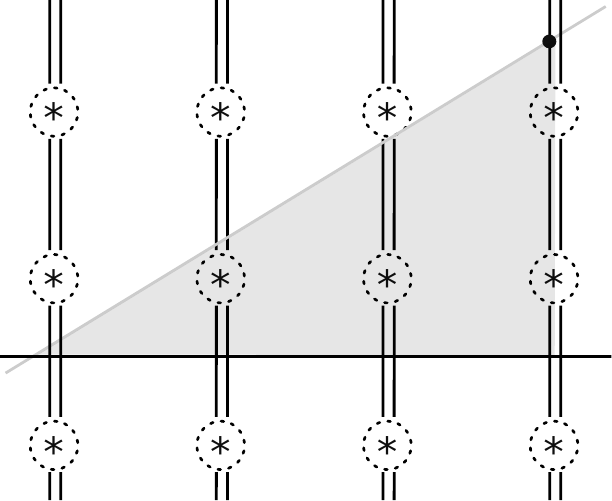}

\caption{The generator $x$ corresponding to the marked intersection point has $A(x) = 2$ since the intersection lies on a vertical segment of height 2, and $k(x) = 2$ since the interior of the shaded triangle covers two marked points.}
\label{fig:triangle}
\end{figure}

\begin{proposition}
\label{prop:grading-change}
For a generator $x \in \HFhat(Y_+,i)$ corresponding to an intersection of $\ell^i_{p/q}$ with a vertical segment in $\hatGamma$, let $A(x)$ and $k(x)$ be the quantities defined above. Then
$$\Delta_{rel}(x) := \mrel(\phi(x)) - \mrel( x ) = 1 - 2 |A(x)| - 4k(x).$$
\end{proposition}
\begin{proof}
We will assume that $A(x) \ge 0$; if $A(x) < 0$, the proof is exactly the same with all pictures rotated 180 degrees and $A(x)$ replaced with $|A(x)|$. We will only work with $\gamma_0$ and the vertical segment containing $x$, and will ignore the rest of $\hatGamma$. Up to perturbing $\hatGamma$ we may assume that the vertical segment in question lies exactly on $\mu$, so that the points $x \in \hatGamma \cap \ell_{p,q}^i$ and $\phi(x) \in \hatGamma \cap \ell_{p,-q}^i$ coincide; this means that $\ell_{p,q}^i$, $\ell_{p,-q}^i$, and $\gamma_0$ form a triangle. In order to compute $\mrel$, we need the grading decoration on $\hatGamma$. We can assume that the set of grading arrows contains an arrow that lies on the right side of $\mu$ and connects $\gamma_0$ to the bottom end of the vertical segment; this grading arrow carries some integer weight $m$. There are two cases to consider, depending on whether the vertical segment containing $x$ is oriented up or down (note that we always assume $\gamma_0$ is oriented rightward). The first case is pictured on the left of Figure \ref{fig:grading-change}; in this case the grading arrow goes from $\gamma_0$ to the right side of the vertical segment. There is a bigon from $x_0^i$ to $x$, shaded dark gray in the figure, which has no cusps, covers $k(x) + A(x)$ punctures, and whose boundary runs over the grading arrow labeled by $m$. By Lemma \ref{lem:maslov-easy-bigons}(b)
$$\mrel(x) = M(x) - M(x_0^i) = -1 + 2k(x) + 2A(x) + 2m.$$
The complement of this region within the triangle formed by $\ell_{p,q}^i$, $\ell_{p,-q}^i$, and $\gamma_0$, shaded light gray in Figure \ref{fig:grading-change}, is a cusped bigon from $\phi(x)$ to $\phi(x_0^i)$. This bigon covers $k(x)$, and its boundary runs over the grading arrow backwards and has a single cusp, at the tail of the grading arrow. It follows from Lemma \ref{lem:maslov-easy-bigons}(c) that
$$\mrel(\phi(x)) =  M(\phi(x)) - M(\phi(x_0^i) = - 2k(x) + 2m.$$
Thus $\Delta_{rel}(x) = \mrel(\phi(x)) - \mrel(x) = 1 - 2A(x) - 4k(x)$, as desired. Note that the label of the grading arrow cancels out and does not end up affecting $\Delta_{rel}(x)$.

In the case that the vertical segment is oriented downward, the grading arrow must go to the left side of the vertical segment to be consistent with the orientations. The right side of Figure \ref{fig:grading-change} shows the modified grading arrow we will use. The only difference is that the boundary of the dark gray bigon from $x_0^i$ to $x$ now has one cusp while the bigon from $\phi(x)$ to $\phi(x_0^i)$ can be drawn with no cusps. This change adds one to $\mrel(x)$ and also adds one to $\mrel(\phi(x))$, so it does not affect $\Delta_{rel}(x)$.
\end{proof}

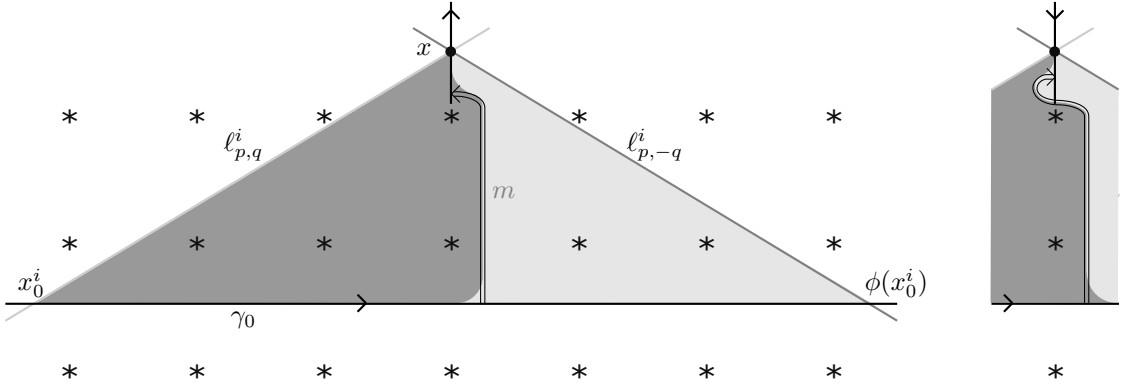
\begin{figure}
\labellist
  \pinlabel {$x$} at 158 124
  \pinlabel {$x_0^i$} at 10 36
  \pinlabel {$\phi(x_0^i)$} at 336 36
  \pinlabel {$\gamma_0$} at 90 22
  \pinlabel {{\color{black}$\ell_{p,q}^i$}} at 90 88
  \pinlabel {{\color{black}$\ell_{p,-q}^i$}} at 245 88
  \pinlabel {{\color{gray}$m$}} at 188 70
  \pinlabel {{\color{gray}$m$}} at 416 70

         \endlabellist
\includegraphics[scale=1]{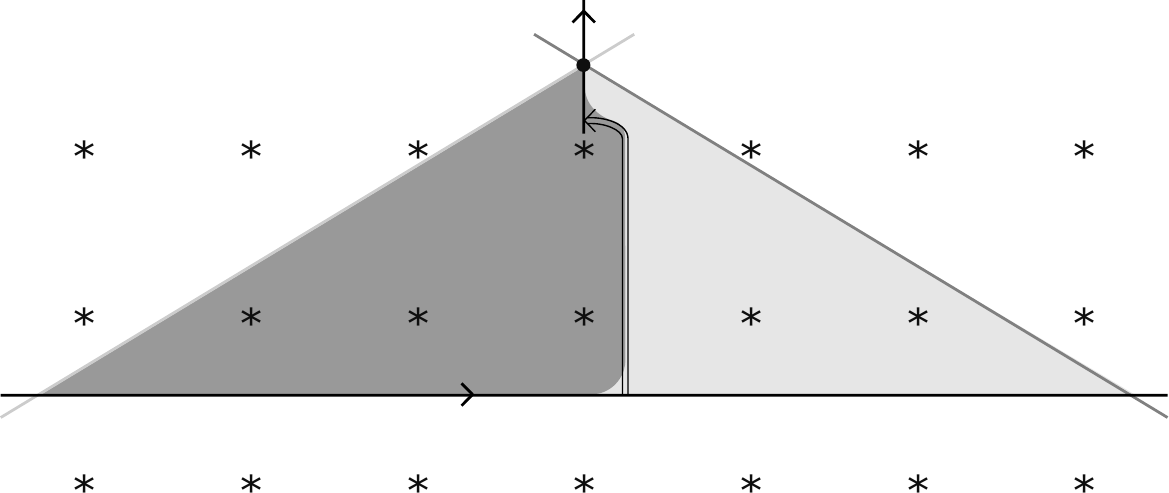} \hspace{1 cm}
\includegraphics[scale=1]{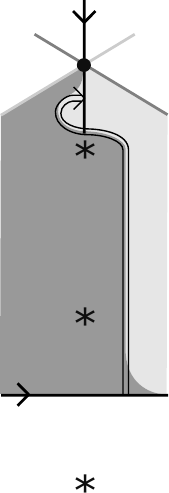}

\caption{Left: A computation of $\Delta_{rel}(x)$. The dark shaded bigon (which has no cusps) can be used to compute $\mrel(x)$, while the lightly shaded region is a cusped bigon used to compute $\mrel(\phi(x))$. Note that the union of these Right: The modification to the diagram needed if the vertical segment is oriented down rather than up.}
\label{fig:grading-change}
\end{figure}

Note that the triangle formed by $\ell_{p,q}^i$, $\ell_{p,-q}^i$, and $\gamma_0$ covers $|A(x)| + 2k(x)$ punctures, so Proposition \label{prop:grading-change} says that $\Delta_{rel}(x)$ is one minus twice the number of punctures covered by this triangle. The number of punctures covered by the triangle is nonnegative. Moreover $k(x) = 0$ if $A(x) = 0$; it follows that $\Delta_{rel}(x) = 1 $ if and only if $A(x) = 0$, and $\Delta_{rel}(x) < 0$ otherwise. An immediate corollary of this is a reproof of a result of Wang:

\begin{corollary}\cite[Theorem 1.3]{Wang}\label{cor:genus1}
If $g(K) = 1$, then $K$ does not admit truly cosmetic surgeries.
\end{corollary}
\begin{proof}
If $K$ admits a truly cosmetic surgery, then by Theorem \ref{thm:NiWu-enhanced} we may assume that the slopes are opposite and $\CFKinfty(K)$ has an isolated generator. Let  $Y_+ = S^3_{p/q}(K)$ and $Y_- = S^3_{-p/q}(K)$, and define $\phi:\HFhat(Y_+) \to \HFhat(Y_-)$ as above. If $g(K) = 1$, then all vertical segments in $\hatGamma$ are at height zero. It follows that every intersection $x$ of any line of slope $p/q$ with a vertical segment of $\hatGamma$ has $A(x) = 0$, and thus has $\Delta_{rel}(x) = 1$. That is, every generator of $\HFhat_{red}(Y_+)$ has its relative grading increase under $\phi$. Since $K$ is nontrivial, $\hatGamma$ has at least one vertical segment so $\HFhat_{red}(Y_+)$ is nontrivial. Thus $\mrelsum(Y_+) > \mrelsum(Y_-)$, and $Y_+ \not\cong Y_-$.
\end{proof}

The above Corollary demonstrates how powerful Proposition \ref{prop:grading-change} can be; we will use this Proposition to derive several more restrictions on the multicurve $\hatGamma$. For example, Corollary \ref{cor:genus1} follows from the fact that if too many vertical segments are at height zero, the sum of all gradings will increase when $\phi$ is applied. On the other hand, it is easy to see that if fewer than half of all vertical segments are at height zero, then the sum of all gradings will decrease. To make this intuition more precise, recall that $n_s$ denotes the number of vertical segments of $\hatGamma$ at height $s$; by the rotational symmetry of $\hatGamma$, $n_{-s} = n_s$. The number of generators $x$ in $\HFhat_{red}(Y_\pm)$ with $A(x) = s$ is $q \cdot n_s$. Since $k(x)$ is always non-negative, Proposition \ref{prop:grading-change} implies that $\Delta_{rel}(x) \le 1-2|s|$ when $A(x) = s$. It follows that if $\mrelsum(Y_-) - \mrelsum(Y_+) = 0$, then
\begin{equation}\label{eq:n0-bound}
n_0 \ge \sum_{s \neq 0} n_s(2|s|-1) = \sum_{s=1}^\infty 2n_s(2s-1) = 2n_1 + 6n_2 + 10n_3 + \cdots
\end{equation}

Armed with this information, we can show that large surgery slopes can never give truly cosmetic surgeries.

\begin{theorem}\label{thm:large-slopes}
Let $K$ be a nontrivial knot in $S^3$. If $Y_+ = S^3_{p/q}(K)$ and $Y_- = S^3_{-p/q}(K)$ are diffeomorphic and $p/q > 1$, then $p/q = 2$, $g(K) = 2$, and $n_0 = 2n_1$.
\end{theorem}
\begin{proof}

Suppose $Y_+ \cong Y_-$. In particular, there is some permutation $\sigma$ on $\Zp$ such that $\HFhat(Y_+, i) \simeq \HFhat(Y_-, \sigma(i))$ as graded vector spaces. Note that $\sigma$ can only permute spin$^c$ structures with the same rank of $\HFhat$, since
$$\HFhat(Y_+, \sigma(i)) \simeq \HFhat(Y_-, \sigma(i)) \simeq \HFhat(Y_+, i)$$
as \emph{ungraded} vector spaces, where the first isomorphism is given by $\phi_{\sigma(i)}$. Because $\tfrac p q > 1$, any line of slope $\tfrac p q$ intersects any vertical segment of $\hatGamma$ at most once. In particular, for $q \le i < p$, the line $\ell_{p,q}^i$ does not hit the vertical segments at height zero at all, while for $0 \le j < q$ the line $\ell_{p,q}^j$ does hit the vertical segments at height zero. Thus we observe that for $j < q$ and $i \ge q$, 
$$\rk \HFhat(Y_\pm, j) \ge 1 + n_0, \quad \text{ and } \quad \rk \HFhat(Y_\pm, i) \le 1 + \sum_{s=1}^\infty 2 n_s.$$

From \eqref{eq:n0-bound}, we have that 
$$\rk \HFhat(Y_\pm, j) - \rk \HFhat(Y_\pm, i) \ge \left(n_0 - \sum_{s=1}^\infty 2n_s\right) \ge \left(\sum_{s=1}^\infty 2n_s(2s-2)\right) = \left( 4n_2 + 8 n_3 + \cdots \right)$$
This difference in dimensions is strictly positive unless $n_s = 0$ for all $s > 1$.

First suppose that $n_s > 0$ for some $s > 1$, so that the above difference is positive for any $j<q$ and $i\ge q$. The spin$^c$ structures of $Y_+$ and $Y_-$ can thus be divided by rank into two subsets, with one set having the $q$ largest dimensions of $\HFhat$ and the other set having the $p-q$ smallest dimensions of $\HFhat$, and for both $Y_+$ and $Y_-$, these subsets are $\{0,\ldots, q-1\}$ and $\{q, \ldots, p-1\}$. The permutation $\sigma$ must fix these two sets; in other words, the first $q$ spin$^c$ structures of $Y_+$ must correspond to the first $q$ to spin$^c$ structures of $Y_-$ under any isomorphism of $\HFhat(Y_+)$ and $\HFhat(Y_-)$. In particular, the sum of the $d$-invariants of these first $q$ spin$^c$ structures must agree. We have
$$\sum_{i=0}^{q-1} d(L(p,q),i) = \sum_{i=0}^{q-1} d(Y_+, i) = \sum_{i=0}^{q-1} d(Y_-, i) = \sum_{i=0}^{q-1} d(L(p,-q),i) = \sum_{i=0}^{q-1} -d(L(p,q),i).$$
It follows that the sum must be zero, but this is impossible by Lemma \ref{lem:first-q}.

Now suppose that $n_s = 0$ for all $s>1$. It follows that $g(K) = 2$, since the maximum $s$ for which $n_s \neq 0$ is $g(K) -1$, and $g(K)$ cannot be 1 by Corollary \ref{cor:genus-one}. Since $p/q > 1$, $k(x) = 0$ for any $x \in \HFhat_{red}(Y_+)$ (i.e. the relevant triangle  does not cover any marked points). By Proposition \ref{prop:grading-change}, 
$$\Delta_{rel}(x) = \begin{cases} \phantom{-}1 & A(x) = 0 \\ -1 & A(x) = \pm 1 \end{cases}$$
These grading changes must cancel when we sum over all generators of $\HFhat_{red}(Y_+)$, which implies that $n_0 = n_1 + n_{-1} = 2n_1$.

There are five possible values of $\rk \HFhat(Y_+,i)$ depending on which vertical segments $\ell_{p,q}^i$ intersects, as shown in the table below (see also Figure \ref{fig:slope-4_3} for an example with a particular slope); this partitions the set of  spin$^c$ structures into five subsets. For each type, we can also compute the net change in relative grading, the sum of $\Delta_{rel}(x)$ over all generators $x$ of $\HFhat_{red}(Y_+, i)$.
\begin{center}
\renewcommand{\arraystretch}{1.2} 
\begin{tabular}{lcll}
  & heights of vertical segments hit by $\ell_{p,q}^i$ & $\rk \HFhat(Y_+, i)$ & $\sum_x \Delta_{rel}(x)$ \\
\toprule
(a) & $\{1,0,-1\}$ & $1 + 4qn_1$ & 0\\
(b) & $\{1,0\}$ or $\{0,-1\}$ & $1 + 3qn_1$ & $qn_1$\\
(c) & $\{1,-1\}$ or $\{0\}$ & $1 + 2qn_1$ & $2qn_1$ or $-2qn_1$ \\
(d) & $\{1\}$ or $\{-1\}$ & $1 + qn_1$ &  $-qn_1$\\
(e) & $\{\}$ & $1$ & 0\\
\midrule
\end{tabular}
\end{center} 
Because $\sigma$ can only permute spin$^c$ structures with the same rank, $\sigma$ must fix the subsets of spin$^c$ structures corresponding to these five types. However, for any spin$^c$ structure of type (b), the total relative grading strictly increases when $\phi$ is applied; it follows that there can be no spin$^c$ structures of type (b). There can also be no spin$^c$ structures of type (d) for similar reasons. But if $p/q > 2$ there is at least one spin$^c$ structure of type (d), namely the one defined by $i = q$, and if $1 < p/q < 2$ there is at least one spin$^c$ structure of type (b), namely $i = 0$. Either case gives a contradiction, so $p/q$ must be 2.
\end{proof}

\begin{figure}
\labellist
  \pinlabel {$\gamma_0$} at 45 91
  \pinlabel {$\ell^0_{4,3}$} at 280 155
  \pinlabel {$\ell^1_{4,3}$} at 280 170
  \pinlabel {$\ell^2_{4,3}$} at 280 180
  \pinlabel {$\ell^3_{4,3}$} at 280 190
         \endlabellist
\includegraphics[scale=1.4]{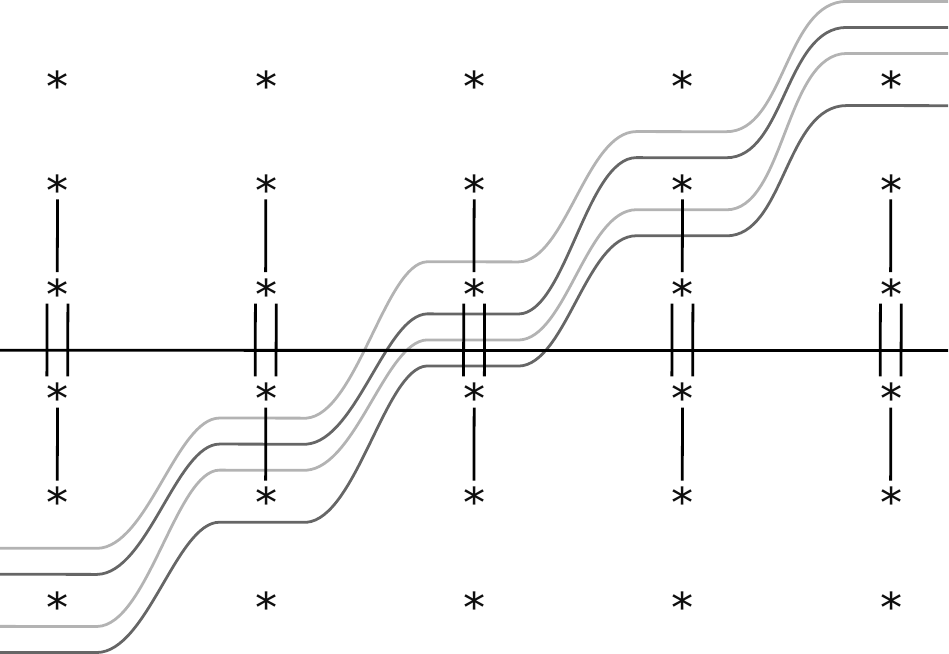}
\vspace{5 mm}
\caption{Intersecting $\hatGamma$ with four lines of slope $4/3$, assuming the knot has genus two with $n_0 = 2n_1$. Each vertical line in the figure represents $n_1$ vertical segments in $\hatGamma$. The spin$^c$ structures $i \in \{0,1,2,3\}$ have types (b), (a), (b), and (c), respectively, as described in the proof of Theorem \ref{thm:large-slopes}.}
\label{fig:slope-4_3}
\end{figure}
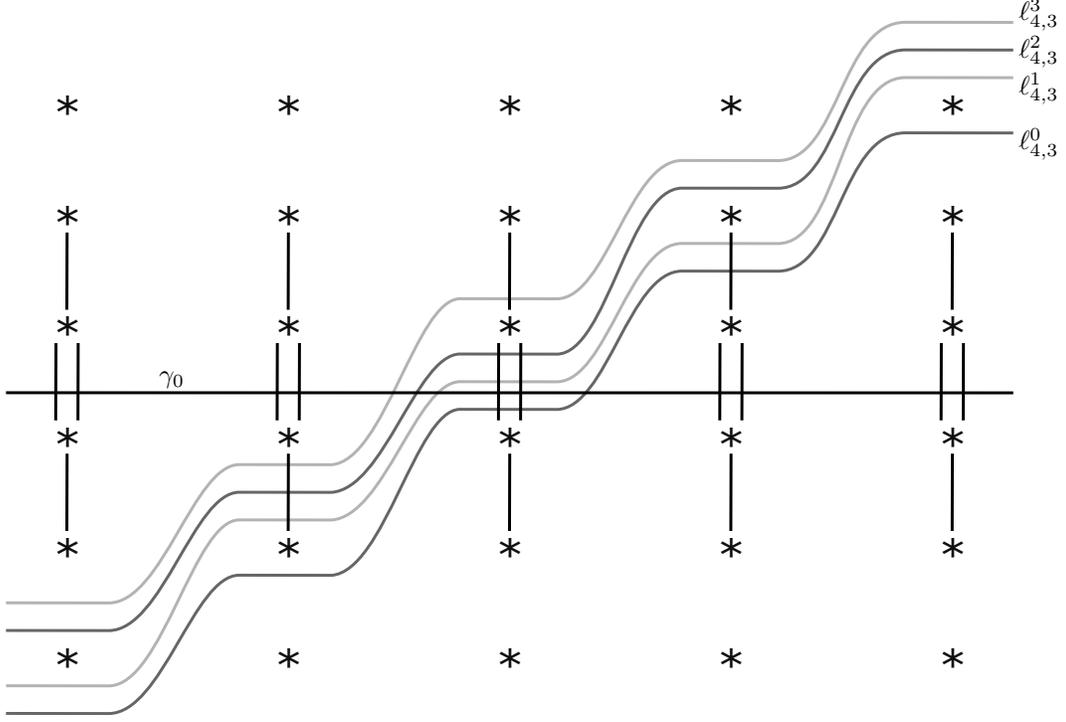

Small slopes can be dealt with in a similar way. Note that for $p/q < 1$, the constant $k(x)$ in Proposition \ref{prop:grading-change} is at least $\tfrac{s(s-1)}{2}$, where $s = |A(x)|$. Moreover, $k(x)$ is strictly larger than this for at least one intersection point $x$ with $|A(x)| = s$, provided there are any intersection points with $|A(x)| = s$. Since we require that $\mrelsum(Y_-) - \mrelsum(Y_+) = 0$, it follows that there exist constants $a_s \in \Q$ with $a_{-s} = a_s$ and $a_s > 2s^2 - 1$ for $s > 0$, such that
\begin{equation}\label{eq:n0-bound2}
n_0 = \sum_{s \neq 0} n_s a_s = \sum_{s=1}^\infty 2n_s a_s > 2n_1 + 14n_2 + 34n_3 + \cdots
\end{equation}
The constants $a_s$ could be computed exactly for any fixed $p/q$, but we will not need this; in fact, we will only need that $a_s > 1$ for all $s>0$.

\begin{theorem}\label{thm:small-slopes}
Let $K$ be a nontrivial knot in $S^3$. If $Y_+ = S^3_{p/q}(K)$ and $Y_- = S^3_{-p/q}(K)$ are diffeomorphic and $p/q < 1$, then $p=1$.
\end{theorem}
\begin{proof}
Suppose to the contrary that $Y_+ \cong Y_-$ and $p \neq 1$. Let $q = mp + r$ with $0 \le r < p$; since $p$ and $q$ are relatively prime, $r > 0$. Any line of slope $p/q$ hits any vertical segment in $\hatGamma$ either $m$ or $m+1$ times; let $c^i_s \in \{m, m+1\}$ denote the number of times the line $\ell_{p, \pm q}^i$ intersects a vertical segment at height $s$. We have
$$\rk \HFhat(Y_\pm, i) = 1 + \sum_{s = -\infty}^\infty c^i_s n_s. $$
If $0 \le i < r$, it is easy to see that $c^i_0 = m+1$, while if $r \le j < p$ then $c^j_0 = m$. Since $c^i_s - c^j_s \ge -1$ for any $s$,
$$\rk \HFhat(Y_\pm, i) - \rk \HFhat(Y_\pm, j) = \sum_{s = -\infty}^\infty (c^i_s - c^j_s)n_s \ge n_0 - \sum_{s=1}^\infty 2n_s > 0,$$
where the last inequality uses \eqref{eq:n0-bound2}. In other words, the first $r$ spin$^c$ structures have ranks strictly bigger than each of the remaining $(p-r)$ spin$^c$ structures. It follows that the permutation $\sigma$ corresponding to the reindexing of spin$^c$ structures under any isomorphism from $\HFhat(Y_+)$ to $\HFhat(Y_-)$ must preserve the first $r$ spin$^c$ structures as a set. In particular,
$$\sum_{i = 0}^{r-1} d(Y_+, i) = \sum_{i = 0}^{r-1} d(Y_-, i).$$
As in the proof of Theorem \ref{thm:large-slopes}, this implies that 
$$0 = \sum_{i = 0}^{r-1} d(L(p,q),i) =  \sum_{i = 0}^{r-1} d(L(p,r),i). $$
But $r \equiv q \equiv -1 \pmod p$, so this is impossible by Lemma \ref{lem:first-q}.
\end{proof}

When we restrict to $p = 1$ we can compute explicit formulas for the net change in relative grading under $\phi$, and this determines $q$ exactly (for a given knot) if $\pm 1/q$ is a pair of truly cosmetic surgery slopes. 

\begin{proposition}\label{prop:ns-ratio}
Suppose $S^3_{1/q}(K) \cong S^3_{-1/q}(K)$. As above, let $n_s$ be the number of vertical segments in $\hatGamma$ at height $s$. Then
$$q = \frac{n_0 + 2\sum_{s=1}^\infty n_s}{4 \sum_{s = 1}^\infty s^2 n_s}.$$
\end{proposition}
\begin{proof}
This is a straightforward consequence of Proposition \ref{prop:grading-change} and the fact that $\sum_x \Delta_{rel}(x) = 0$. Note that for slope $1/q$ there are $q$ intersections of line $\ell_{1,q}$ with any vertical segment at height $s > 0$, and if we label these intersections by $i = 0, \ldots, q-1$, the constant $k(x)$ for the $i$th intersection is
$$\left[q(s-1) + q(s-2) + \cdots + q \right] + is = \frac{qs(s-1)}{2} + is,$$
and the sum over all $q$ of these points of $4k(x) + 2A(x) - 1$ is
$$4\left[\frac{q^2s(s-1)}{2} + \frac{sq(q-1)}{2}\right] +2qs - q = 2q^2s^2 - q.$$
By symmetry the contribution to $\sum_x \Delta_{rel}(x)$ of a vertical segment at height $-s$ is he same as that of a vertical segment at height $s$. Each vertical segment at height 0 contributes $q$ intersection points, each with $\Delta_{rel}(x) = 1$. The condition that $\sum_x \Delta_{rel}(x) = 0$ can now be stated as
$$qn_0 - 2\sum_{s = 1}^\infty (2q^2s^2 - q) n_s = 0;$$
solving this equation for $q$ gives the desired result.
\end{proof}
Note that since $n_s = 0$ for $|s| \ge g(K)$, the infinite sums above can be truncated for any particular example.

Another powerful consequence of Proposition \ref{prop:grading-change} is a bound on $q$ and the genus of $K$. Theorems \ref{thm:large-slopes} and \ref{thm:small-slopes} and Proposition \ref{prop:ns-ratio} rely on the fact the sums of relative gradings $\mrelsum(Y_+)$ and $\mrelsum(Y_-)$ should agree, and thus for every generator $x$ of $\HFhat_{red}(Y_+)$ with $\Delta_{rel}(x) = -n < 0$, there must be $n$ generators $y_1, \ldots, y_n$ with $\Delta_{rel}(y_i) = 1$. But the set of relative gradings is an invariant, not just its sum, so in fact $\mrel(Y_+) = \mrel(Y_-)$ as multisets. This lets us say more:
\begin{lemma}\label{lem:chain}
Suppose $Y_+ \cong Y_-$. If $\HFhat_{red}(Y_+)$ contains a generator $x$ with $\Delta_{rel}(x) = -n < 0$ and $\mrel(x) = m$, then it must contain generators $y_1, \ldots, y_n$ with $\Delta_{rel}(y_i) = 1$ and $\mrel(y_i ) = m-i$.
\end{lemma}
\begin{proof}
If $\psi:\HFhat(Y_+) \to \HFhat(Y_-)$ is a grading preserving isomorphism, then $\psi^{-1}\circ\phi$ determines a permutation on the multiset of gradings $\mrel(Y_+)$. This permutation takes an element $m$ of $\mrel(Y_+)$ (corresponding to the grading of the generator $x$) to $m-n$. The cycle containing this element must eventually return to $m$, but at each step the relative grading can increase by at most one by Proposition \ref{prop:grading-change}. Thus for each $m-i$ between $m-n$ and $m-1$, $m-i$ appears in the cycle followed by $m-i+1$; each such element $m-i$ corresponds to $\mrel$ for some generator with positive $\Delta_{rel}$, and we take this generator to be $y_i$.
\end{proof}
We now need to relate $\mrel$ of generators of $\HFhat(Y_+)$ with $\Delta_{rel} = 1$ to the knot Floer homology of $K$. Recall that the $\delta$-grading refers to the difference between the Alexander and Maslov gradings.

\begin{lemma}\label{lem:CFK-delta-gradings}
If $x$ is a generator of $\HFhat(Y_+)$ coming from a vertical segment of height 0 and $\mrel(x) = m$, then there is an intersection point in $\hatGamma \cap \mu$ (i.e. a generator of $\HFKhat(K)$) with $\delta$-grading at least $-m-1$, and there is one with $\delta$-grading at most $-m$.
\end{lemma}
\begin{proof}
Let $V$ be the vertical segment in $\hatGamma$ containing $x$. We may assume there is a grading arrow in $\hatGamma$ moving upward from $\gamma_0$ to $V$, as shown in Figure \ref{fig:height0-gradings}(a) if $m$ is odd and \ref{fig:height0-gradings}(d) if $m$ is even. Consider first the case that $m$ is odd, with $m = 2k-1$; then the arrow has weight $k$ and approaches $V$ from the right. We will assume that $\hatGamma$ is in minimal position with $\mu$ and note that $V$ lies either to the right of $\mu$, as in Figure \ref{fig:height0-gradings}(b), or to the left of $\mu$, as in Figure \ref{fig:height0-gradings}(c). In either case, let $x'$ and $x''$ in $\hatGamma\cap\mu$ be the first intersections with $\mu$ when following $\hatGamma$ upward and downward, respectively, from $V$. Note that $A(x') \ge 0$ and $A(x'') \le 0$.

Let $x_0$ denote the intersection of $\gamma_0$ with $\mu$, which by definition has Maslov grading $M(x_0) = 0$. If $V$ is to the right of $\mu$, there is a bigon from $x_0$ to $x'$ whose boundary traverses the grading arrow and which passes $A(x')$ marked points along $\mu$; from this and \eqref{eq:maslov-shift2} in Section \ref{sec:bifiltered-complexes} we compute
$$M(x') = M(x_0) + 2k + 2A(x') - 1 = m + 2A(x'),$$
$$\delta(x') = A(x') - M(x') = -m - A(x') \le -m.$$
The bigon from $x''$ to $x'$, which passes $A(x')-A(x'')$ marked points along $\mu$, gives
$$M(x'') = M(x') + 1 -2(A(x')-A(x'')) = m + 2A(x'') + 1,$$
$$\delta(x'') = A(x'') - M(x'') = -m - 1 - A(x'') \ge -m - 1.$$
If instead $V$ lies to the left of of $\mu$ then a bigon from $x_0$ to $x''$ gives
$$M(x'') = 2k - 1 = m, \hspace{1cm} \delta(x'') = A(x'') - m \le -m,$$
and a bigon from $x'$ to $x''$ gives
$$M(x') = M(x'') + 1 = m + 1, \hspace{1cm} \delta(x') = A(x') - m - 1 \ge -m - 1.$$

The case that $m = 2k$ is similar. The grading arrow must approach $V$ from the left, as in Figure \ref{fig:height0-gradings}(d), which adds a cusp to the boundary the bigons considered above which run over the grading arrow. This increases $M(x')$ and $M(x'')$ by one (computing in terms of $k$), which decreases the $\delta$ gradings by one, but $m$ is also increased from $2k-1$ to $2k$, so the conclusion still holds.
\end{proof}
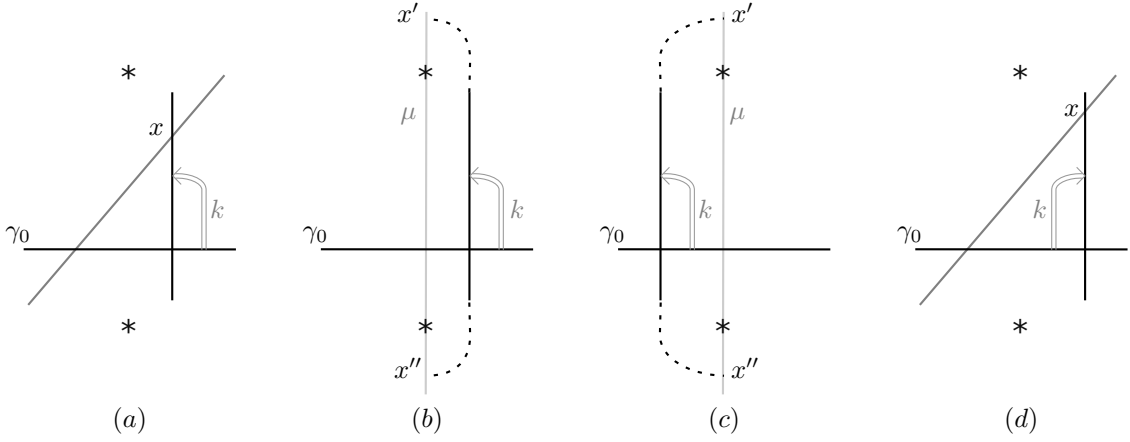
\begin{figure}
\labellist
  \pinlabel {$x$} at 50 100
  \pinlabel {$x$} at 395 107
  \pinlabel {$x'$} at 145 144
  \pinlabel {$x''$} at 145 10
  \pinlabel {$x'$} at 271 144
  \pinlabel {$x''$} at 272 10
  \pinlabel {$\gamma_0$} at -2 60
  \pinlabel {$\gamma_0$} at 110 60
  \pinlabel {$\gamma_0$} at 222 60
  \pinlabel {$\gamma_0$} at 334 60

  \pinlabel {{\color{gray}$k$}} at 73 70
  \pinlabel {{\color{gray}$k$}} at 186 70
  \pinlabel {{\color{gray}$k$}} at 257 70
  \pinlabel {{\color{gray}$k$}} at 383 70
  \pinlabel {{\color{gray}$\mu$}} at 145 105
  \pinlabel {{\color{gray}$\mu$}} at 269 105
  
  \pinlabel {$(a)$} at 40 -10
  \pinlabel {$(b)$} at 152 -10
  \pinlabel {$(c)$} at 264 -10
  \pinlabel {$(d)$} at 376 -10

         \endlabellist
\includegraphics[scale=1]{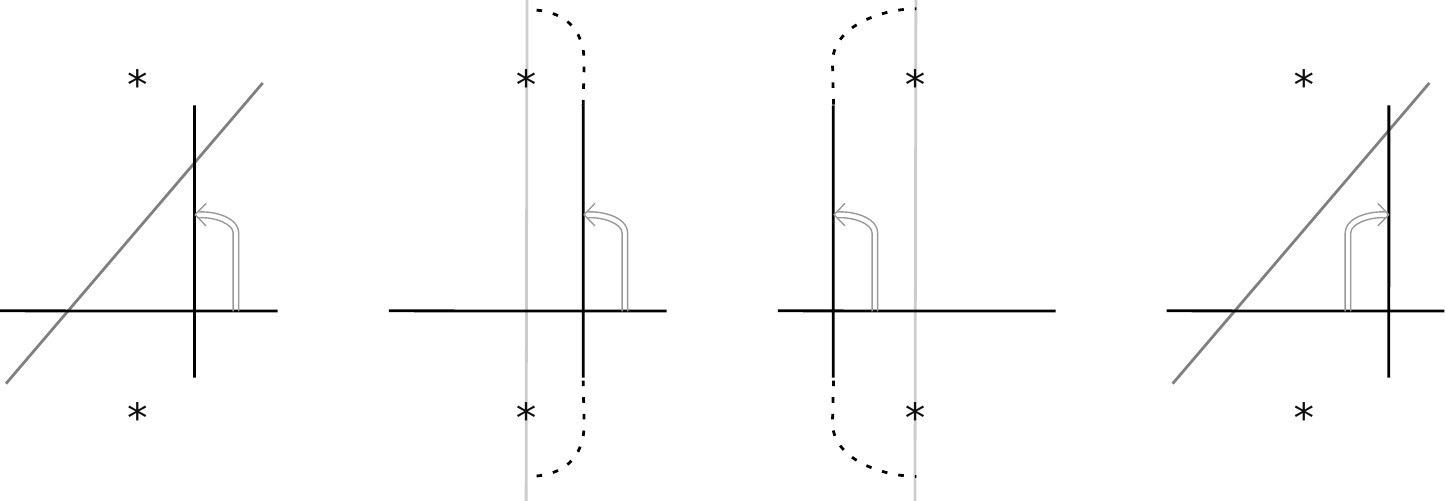}
\vspace{5 mm}

\caption{(a) A grading arrow such that $\mrel(x) = m = 2k-1$; (b,c) $x'$ and $x''$, the first intersections with $\mu$ above and below the vertical segment containing $x$; (d) a grading arrow such that $\mrel(x) = m = 2k$. }
\label{fig:height0-gradings}
\end{figure}

We are now ready to prove the promised bounds on $q$ and $g(K)$. Recall that $th(K)$ denotes the Heegaard Floer thickness of $K$; that is,
$$th(K) = \max\{ \delta(x) | x \in \HFKhat(K) \} - \min\{ \delta(x) | x \in \HFKhat(K) \}$$

\begin{theorem} \label{thm:thickness-bound}
Let $K$ be a nontrivial knot in $S^3$ of genus $g$. If $Y_+ = S^3_{1/q}(K)$ and $Y_- = S^3_{-1/q}(K)$ are diffeomorphic, then
$$ th(K) \ge 2qg(g-1) - 2g.$$
\end{theorem}
\begin{proof}
The maximum height attained by $\hatGamma$ is $g$ so there are at least two vertical segments at height $g-1$, one on either side of a maximum of $\hatGamma$. Consider the highest intersection point of each of these two vertical segments with any $\ell_{1,q}^i$; call these $x$ and $x'$ as shown in Figure \ref{fig:thickness-triangle}. By construction, $A(x) = A(x') = g-1$. Counting the marked points in the closure of the triangle in Figure \ref{fig:thickness-triangle} and removing those on the boundary gives
$$k(x) = k(x') = \left[ q(g-1) + q(g-2) + \cdots + q(1) \right] - (g-1) = q\frac{g(g-1)}{2} - (g-1).$$
Let $d$ denote $-\Delta_{rel}(x) = -\Delta_{rel}(x')$. We have
$$d = 4k(x) + 2A(x) - 1 =   4\left( \frac{qg(g-1)}{2} - (g-1) \right) + 2(g-1) - 1 = 2qg(g-1) - 2g + 1.$$
Let $m$ be $\mrel(x)$; the small bigon from $x$ to $x'$ covering one puncture tells us that $\mrel(x') = \mrel(x) + 1 = m + 1$. By Lemma \ref{lem:chain}, there are generators $y_0, \ldots, y_d$ of $\HFhat_{red}(Y_+)$ with $\mrel(y_i) = m - i$ and $\Delta_{rel}(y_i) = +1$. Applying Lemma \ref{lem:CFK-delta-gradings} to $y_0$ and $y_d$, we find that $\HFKhat(K)$ contains generators $z_1$ and $z_2$ with $\delta(z_1) \le -m$ and $\delta(z_2) \ge -m - 1 + d$. Thus $th(K) \ge d-1$.
\end{proof}
\begin{figure}
\labellist
  \pinlabel {$x$} at 186 88
  \pinlabel {$x'$} at 166 92
  \pinlabel {$\gamma_0$} at 200 15

  \pinlabel {{\color{gray}$\ell_{1,q}^i$}} at 100 60

         \endlabellist
\includegraphics[scale=1]{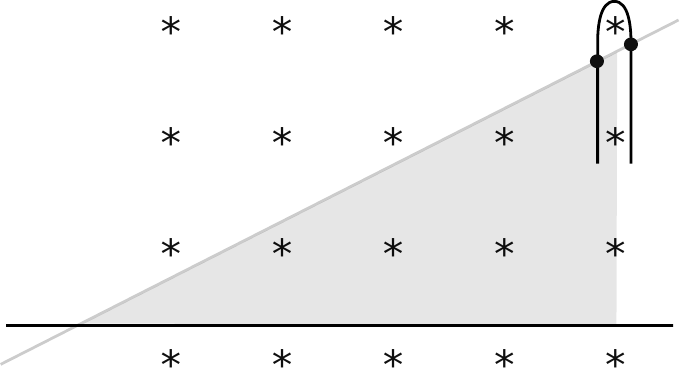}

\caption{Generators $x$ and $x'$ with minimal $\Delta_{rel}$ in  $\HFhat_{red}(Y_+)$ for surgery slope $1/q$. Here $q = 2$ and $g(K) = 3$.}
\label{fig:thickness-triangle}
\end{figure}
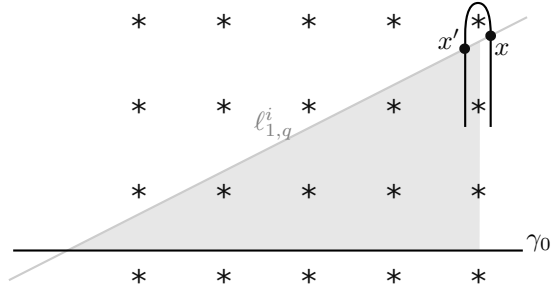

Theorems \ref{thm:large-slopes} and \ref{thm:small-slopes}, Proposition \ref{prop:ns-ratio}, and Theorem \ref{thm:thickness-bound} combine to give Theorem \ref{thm:main} in the introduction.

\section{Explicit obstructions in terms of $\hatGamma$}\label{sec:explicit-obstructions}

\subsection{General constraints on $\hatGamma$}
In addition to Theorem \ref{thm:main}, it is helpful to have explicit conditions on a knot $K$, in terms of its knot Floer invariant, that ensure that $K$ admits no truly cosmetic surgeries at all. Several such conditions are already implicit in what has been discussed so far. One condition comes from Theorem \ref{thm:NiWu-enhanced}, namely that the curve $\gamma_0$ is horizontal. Another condition follows from the bounds in Theorem \ref{thm:main}: if $g = g(K) \neq 2$ and $th(K) + 2g < 2g(g-1)$, then $K$ admits no truly cosmetic surgeries. We have also seen that Proposition \ref{prop:grading-change} and the fact that $\sum_x \Delta_{rel}(x) = 0$ places constraints on the numbers $n_s$ of vertical segments in $\hatGamma(K)$ at height $s$, including the inequality \eqref{eq:n0-bound}. In particular, if fewer than half of all vertical segments occur at height 0, then $K$ cannot admit any truly cosmetic surgeries. For the slopes $\tfrac 1 q$ the inequality \eqref{eq:n0-bound} can be improved to an equation that can be solved for $q$, giving rise to Proposition \ref{prop:ns-ratio}. This places a further constraint on the existence of truly cosmetic surgeries which is implicit in Theorem \ref{thm:main}: if the quantity
$$\frac{n_0 + 2\sum_{s=1}^\infty n_s}{4 \sum_{s = 1}^\infty s^2 n_s}$$
is not a positive integer for a given knot $K$, then $K$ does not admit truly cosmetic surgeries.

To arrive at Proposition \ref{prop:ns-ratio} and the other constraints mentioned above, we only assumed that $S^3_r(K)$ and $S^3_{-r}(K)$ have the same $\emph{sum}$ of all relative gradings. By considering the \emph{set} of all relative gradings, we could impose further constraints on $K$. For example, in the case that $r=2$, it is not enough to have two vertical segments at height 0 for each vertical segment at height 1, we also require that vertical segments at height 0 give rise to generators in $S^3_r(K)$ that have grading one less than those coming from the vertical segment at height 1. Unfortunately it is difficult to state such conditions purely in terms of the knot Floer homology of $K$; this is partly because there is not a perfect correspondence between gradings of elements of the surgery coming from a vertical segment and gradings in knot Floer homology. 

We will see that this difficulty can be overcome in a particular special case. We remark that while it might be possible to state some additional (likely messy) conditions on $\hatGamma(K)$ in the general case, it is not worth doing so. In practice, to achieve this fine an obstruction on $K$ it is easiest to simply compute absolutely graded $\HFplus$ for each of the finitely many surgery pairs allowed by Theorem \ref{thm:main} and check if they agree for any pair. In this way, we can extract the maximum information from Heegaard Floer homology: for any knot we can always either rule out cosmetic surgeries or conclude that Heegaard Floer homology can not rule out some potential pair. This can be thought of as a condition on the knot Floer homology of $K$ (by the surgery formula, $\HFplus$ of the relevant surgeries is determined by knot Floer homology), though it is a condition that requires some computation to check.

\subsection{Further constraints for simple figure eight curves}  There is one situation where it is convenient to state additional constraints purely in terms of knot Floer homology, and that is when the underlying curve set for $\hatGamma$ consists only of $\gamma_0$ and simple figure eight curves. As mentioned in Remark \ref{rmk:only-figure-eights}, this is common in practice.

A simple figure eight component $\gamma_i$ of $\curveset$ intersects $\mu$ four times, corresponding to four generators of the knot Floer homology of $K$. These generators all have the same $\delta$ grading, so it makes sense to talk about the $\delta$ grading of the curve $\gamma_i$. The height of a simple figure eight component is the height at which it is centered, which is the Alexander grading of two of the four generators. Let $e_s^d$ denote the number of simple figure eight components in $\curveset$ at height $s$ with $\delta$-grading $d$, and let $e_s = \sum_{d\in\Z} e_s^d$ be the total number of simple figure eights at height $s$. Each simple figure eight curve contributes two vertical segments at height $s$, so if we assume $\gamma_0$ is horizontal and all other $\gamma_i$'s are simple figure eights then $n_s = 2e_s$. We will assume the self intersection in a simple figure eight curve $\gamma_i$ occurs below the vertical segments, as in Figure \ref{fig:figure-eights}, and with this understanding we will refer to the \emph{left} and \emph{right} vertical segments coming from $\gamma_i$. The relative grading of $\gamma_i$ is determined by its $\delta$-grading. In particular, if $\gamma_i$ is a simple figure eight at height $s\ge 0$ with $\delta$-grading $d$, we can add a consistent grading arrow from $\gamma_0$ to one of the vertical segments in $\gamma_i$ as shown in Figure \ref{fig:figure-eights}(a,b); the arrow passes to the right of any marked points up to height $s$ and ends on the right vertical segment and carries the weight $\tfrac{-d-s}{2}$ if $d+s$ is even, or it ends on the left vertical segment and carries the weight $\tfrac{1-d-s}{2}$ if $d+s$ is odd. The case of height $s < 0$ is similar, except that the arrow stays to the left of the marked points and $s$ is replaced with $|s|$ in the arrow weights.

\begin{remark}
This gives an alternative (and much simpler) way of encoding grading information in $\hatGamma(K)$ in the case that all curves other than $\gamma_0$ are simple figure eights: instead of decorating the set of curves with a collection of grading arrows, we can simply decorate each curve other than $\gamma_0$ with an integer, its $\delta$-grading.
\end{remark}

We now relate the $\delta$-grading $d$ of a simple figure eight component $\gamma_i$ at height $s$ to the relative grading of generators of $\HFhat(Y_+)$ coming from $\gamma_i$. Note that for each generator $x$ coming from the right vertical segment of $\gamma_i$, there is a corresponding generator $x'$ coming from the left vertical segment, as shown in Figure \ref{fig:figure-eights}(c). These intersection points are connected by a small bigon covering one marked point, so that $\mrel(x') = \mrel(x) + 1$. It is clear that $k(x) = k(x')$ and $A(x) = A(x') = s$, and thus that $\Delta_{rel}(x) = \Delta_{rel}(x') = 1 - 4k(x) -2s$. We calculate that $\mrel(x) = -1 + 2k(x) + s - d$, regardless of the parity of $s+d$. For any given slope $r$, it is possible to compute the $k(x)$ for all intersection points of lines of slope $r$ and state an obstruction to $S^3_r(K)$ and $S^3_{-r}(K)$ agreeing purely in terms of the quantities $e^d_s$ for $K$; we will only do this for the slope $r = 1$, where every intersection point $x$ coming from a simple figure eight at height $s$ has $k(x) = \tfrac{s(s-1)}{2}$.

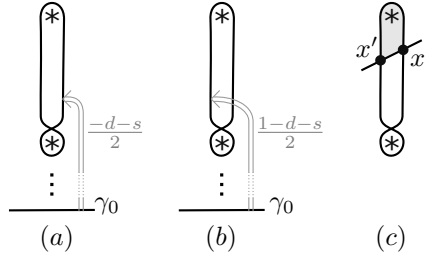
\begin{figure}
\labellist
  \pinlabel {$x$} at 154 57
  \pinlabel {$x'$} at 135 62
  
  \pinlabel {$\gamma_0$} at 103 2
  \pinlabel {$\gamma_0$} at 37 2

  \pinlabel {{\color{gray}$\frac{1-d-s}{2}$}} at 106 30
  \pinlabel {{\color{gray}$\frac{-d-s}{2}$}} at 40 30

  \pinlabel {$(a)$} at 18 -10
  \pinlabel {$(b)$} at 80 -10
  \pinlabel {$(c)$} at 144 -10

         \endlabellist
\includegraphics[scale=1]{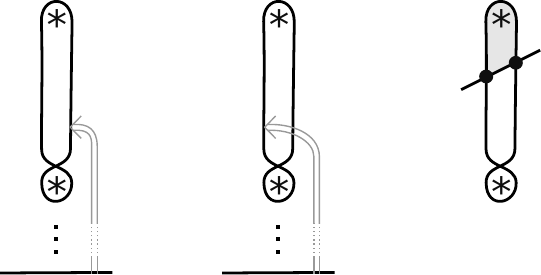}
\vspace{5 mm}
\caption{(a) A grading arrow for a simple figure eight component at height $s$ with $\delta$-grading $d$ if $d+s$ is even; (b) a grading arrow if $d+s$ is odd; (c) a pair $x$ and $x'$ of intersection points representing generators of $\HFhat(Y_+)$, with $\mrel(x') = \mrel(x) + 1$.}
\label{fig:figure-eights}
\end{figure}

\begin{proposition}\label{prop:figure-eight-contraints-slope1}
Suppose that $\hatGamma(K)$ consists only of $\gamma_0$ and simple figure eight curves, and let $e^d_s$ be the number of simple figure eight components with height $s$ and $\delta$-grading $d$. If $S^3_{+1}(K) \cong S^3_{-1}(K)$, then $\gamma_0$ is horizontal and for every $D \in \Z$, we must have
$$ e_0^D = \sum_{s\neq 0} \sum_{d = -D-s^2+1}^{-D+s^2-1} e^d_s.$$
\end{proposition}
\begin{proof}
We already know that $\gamma_0$ is horizontal by Theorem \ref{thm:NiWu-enhanced}. We now consider the graded contribution of the simple figure eight components to each surgery. It is enough to consider only one generator of $\HFhat_{red}(S^3_{+1}(K))$ coming from each simple figure eight component, the one coming from the right vertical segment, since the generators coming from the left vertical segments behave exactly the same with the grading shifted up by one. With this in mind, a simple figure eight component with height $s$ and $\delta$-grading $d$ contributes a generator $x$ with $\mrel(x) = s^2 - 1 -d$ and $\Delta_{rel}(x) = 1 - 2s^2$. By Lemma \ref{lem:chain}, this must be counteracted by a chain of generators coming from height zero vertical segments with gradings ranging from $-d -s^2 $ up to $-d + s^2-2$. These must come from a chain of height 0 simple figure eight components with $\delta$-gradings $-d -s^2 +1 , -d -s^2 +2, \ldots, -d + s^2 - 1$ (we are still considering only the right vertical segment of each figure eight). All height zero figure eight components must be accounted for in one of these chains, and we see that there is a contribution to $e^D_0$ for each figure eight component with height $s \neq 0$ and grading $d$ with $d \ge -D-s^2+1$ and $d \le -D+s^2-1$.
\end{proof}
Restricting to knots with small thickness gives the following condition:
\begin{proposition}\label{prop:figure-eight-constraints-thin}
Suppose, as in the previous Proposition, that $\hatGamma(K)$ consists only of $\gamma_0$ and simple figure eight components, and suppose $th(K) <  4$. If $K$ admits any truly cosmetic surgery, then $e_s^d = 0$ for all $|s| > 1$, and $e_0^d = e_1^d + e_{-1}^d = 2e_1^d$ for every $\delta$-grading $d$.
\end{proposition}
\begin{proof}
By Theorem \ref{thm:main}, we know that $g(K) = 2$, which implies $e_s^d = 0 $ if $|s| \ge 2$. We also know that the only possible truly cosmetic surgeries have slopes $\pm 1$ or $\pm 2$. In the first case, we apply Proposition \ref{prop:figure-eight-contraints-slope1} to get the desired result. For the latter case, note that intersections $x$ between lines of slope $2$ and vertical segments at height $1$ or $-1$ still have $k(x) = 0$, and so $\mrel(x)$ and $\Delta_{rel}(x)$ are the same as in the slope 1 case, and the reasoning in the proof of Proposition \ref{prop:figure-eight-contraints-slope1} applies.
\end{proof}
There is a sort of Heegaard Floer converse to this statement: if $g(K) = 2$, and $e_0^d = 2e_1^d$ for all gradings $d$, then Heegaard Floer homology can not distinguish either pair---that is, $\HFhat(S^3_r(K))$ and $\HFhat(S^3_{-r}(K))$ are isomorphic as absolutely graded vector spaces for $r \in \{1,2\}$. We compute this explicitly for the example of $9_{44}$ below, which serves as a model computation for the general case. One might hope that upgrading to $\HFplus$ would help in this situation, but it does not. In fact, when $\gamma_0$ is horizontal and all other curves in $\hatGamma(K)$ are simple figure eight components, $\HFplus$ is determined by $\HFhat$ for any surgery on $K$.

In the case of thin knots this condition can be given purely in terms of the Alexander polynomial, as stated in the introduction.
\begin{proof}[Proof of Theorem \ref{thm:alternating}]
We use Proposition \ref{prop:figure-eight-constraints-thin}, though since $K$ is thin there is only one occupied $\delta$-grading, which must be 0 if $\gamma_0$ is horizontal. For $K$ to admit truly cosmetic surgeries, we must have some number $n$ of simple figure eights at height 1 (and at height $-1$, by symmetry), and $2n$ simple figure eight components at height 0. It easy to compute $\Delta_K(t)$ from this information and see that it has the desired form. (Conversely, for thin knots $\hatGamma(K)$ is determined by $\Delta_K(t)$ and $\sigma(K)$; for  $\sigma(K) = 0$ and $\Delta_K(t)$ as in the conclusion of the Theorem, it is easy to check the $\gamma_0$ is horizontal, $e_1^0 = e_{-1}^0 = n$, and $e_0^0 = 2n$, and thus Heegaard Floer homology does not distinguish $\pm 1$ surgeries or $\pm 2$ surgeries.)
\end{proof}

\subsection{Unobstructed knots}
We conclude this section by demonstrating that Theorem \ref{thm:main} cannot be substantially improved using Heegaard Floer homology alone. We first note that there exist knots for which Heegaard Floer homology does not distinguish $\pm 1$ surgeries or $\pm 2$ surgeries. For example, consider the knot $9_{44}$ shown in Figure \ref{fig:9_44}(a). This example appeared in \cite{OzSz:rational-surgeries}, where it was first observed that $\HFplus(S^3_1(9_{44})) \cong \HFplus(S^3_{-1}(9_{44}))$. It turns out that this example is representative of all currently known examples for which Heegaard Floer homology does not obstruct truly cosmetic surgeries, so we now examine this example in detail. The knot Floer invariant $\hatGamma(9_{44})$ is shown in Figure \ref{fig:9_44}(b) (see also Figure \ref{fig:curve-examples}); $\gamma_0$ is horizontal and there are four simple figure eight components, with $e_1^0 = e_{-1}^0 = 1$ and $e_0^0 = 2$. 


\begin{figure}
\labellist
  \pinlabel {$(a)$} at 30 -10
  \pinlabel {$(b)$} at 110 -10
  \pinlabel {$(c)$} at 185 -10
  \pinlabel {$(d)$} at 260 -10
  
  \pinlabel {$\gamma_0$} at 92 73
  \pinlabel {$\gamma_1$} at 102 140
  \pinlabel {$\gamma_2$} at 93 95
  \pinlabel {$\gamma_3$} at 109 95
  \pinlabel {$\gamma_4$} at 102 35

\tiny
  \pinlabel {$x_0$} at 169 74

  \pinlabel {$b_1$} at 178 142
  \pinlabel {$a_1$} at 191 144
  \pinlabel {$b_2$} at 169 86
  \pinlabel {$a_2$} at 200 105
  \pinlabel {$b_3$} at 177 93
  \pinlabel {$a_3$} at 192 96
  \pinlabel {$b_4$} at 178 47
  \pinlabel {$a_4$} at 191 49
  
  \pinlabel {$x_0$} at 244 74

  \pinlabel {$b_1$} at 253 144
  \pinlabel {$a_1$} at 266 146
  \pinlabel {$b_2$} at 244 88
  \pinlabel {$a_2$} at 275 108
  \pinlabel {$b_3$} at 252 96
  \pinlabel {$a_3$} at 267 100
  \pinlabel {$b_4$} at 253 48
  \pinlabel {$a_4$} at 266 50

         \endlabellist
\includegraphics[scale=1.5]{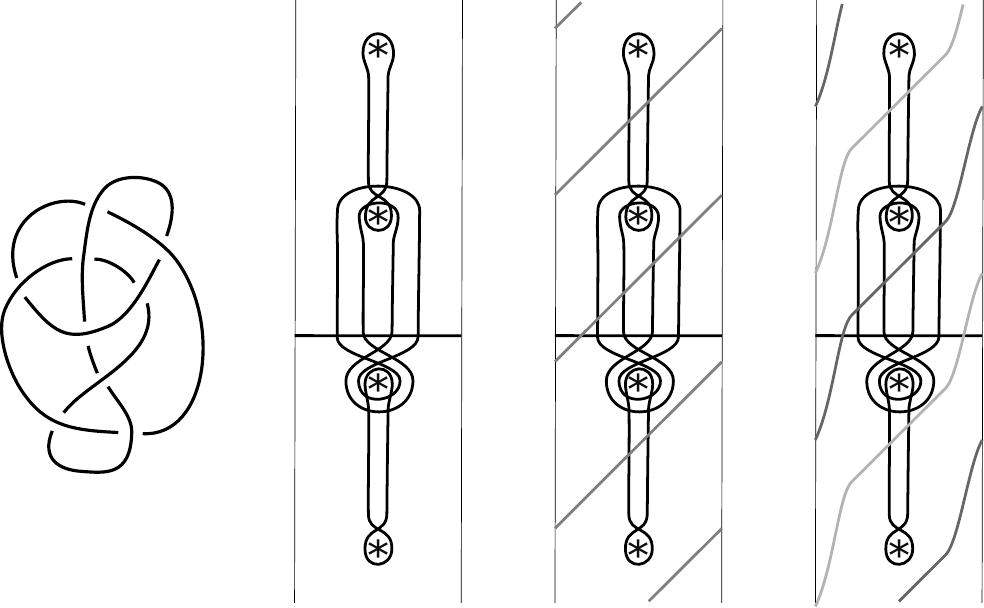}
\vspace{5 mm}
\caption{(a) The knot $9_{44}$; (b) the invariant $\hatGamma(9_{44})$ (all simple figure eight components have $\delta$-grading 0); (c) the computation of $\HFhat(S^3_1(9_{44}))$; (d) the computation of $\HFhat(S^3_2(9_{44}))$.}
\label{fig:9_44}
\end{figure}
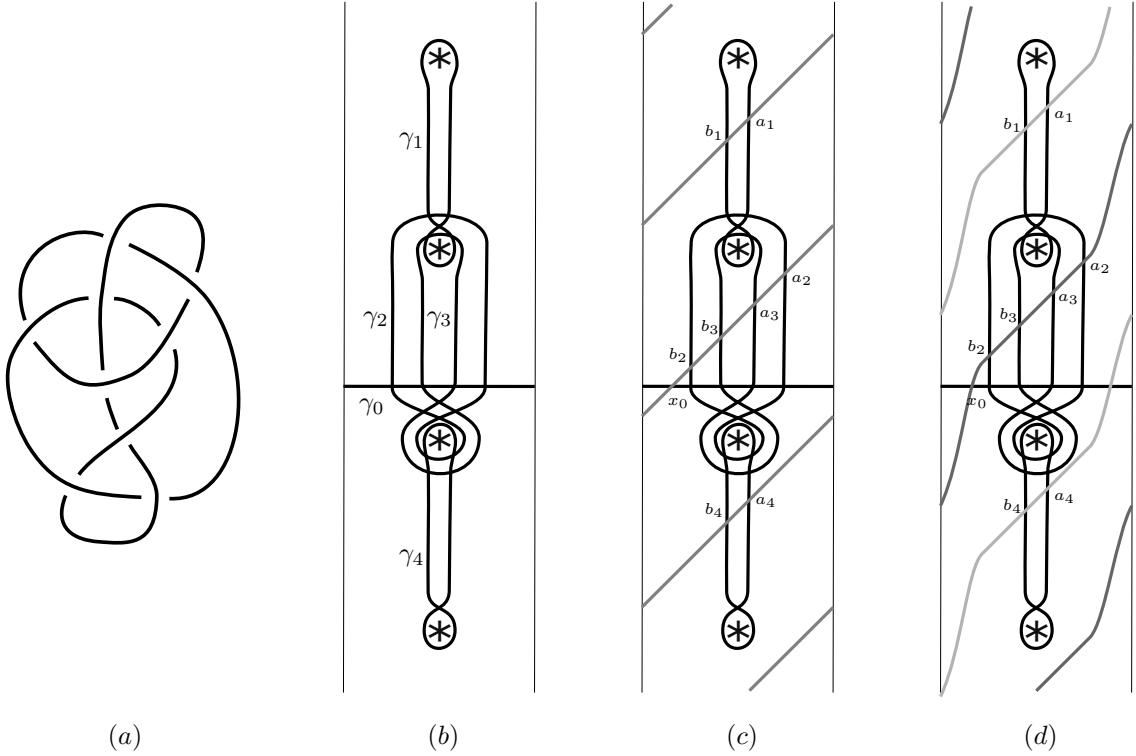

Let $Y_\pm$ denote $S^3_{\pm 1}(9_{44})$; we compute $\HFhat(Y_+)$ using part $(c)$ of Figure \ref{fig:9_44}, which shows the intersection of $\hatGamma(9_{44})$ with a line of slope $+1$. It is easy to see that, in addition to the generator $x_0$ coming from $\gamma_0$, $\HFhat(Y_+)$ has a pair of generators $(a_i, b_i)$ coming from each of the simple figure eight curves, with $\mrel(b_i) = \mrel(a_i) + 1$. For the two extremal figure eights (that is, for $i \in \{1, 4\}$), we have $\mrel(a_i) = 0$, $\mrel(b_i) = 1$, and $\Delta_{rel}(a_i) = \Delta_{rel}(b_i) = -1$. For the two figure eights at height zero (that is, for $i \in \{2,3\}$), we have $\mrel(a_i) = -1$, $\mrel(b_i) = 0$, and $\Delta_{rel}(a_i) = \Delta_{rel}(b_i) = 1$. Applying $\phi$ with the given grading changes, we see that $\HFhat(Y_-)$ has $\mrel(\phi(x_0)) = 0$, $\mrel(\phi(a_i)) = -1$ for $i \in \{1,4\}$, $\mrel(\phi(b_i)) = 0$ for $i \in \{1,4\}$, $\mrel(\phi(a_i)) = 0$ for $i \in \{2,3\}$, and $\mrel(\phi(b_i)) = 1$ for $i \in \{2,3\}$. Thus $\HFhat(Y_+)$ and $\HFhat(Y_-)$ agree as relatively graded vector spaces. The absolute grading agrees as well since $d(Y_+) = 0 = d(Y_-)$. Note that the computation is similar for any genus two knot for which $\hatGamma$ contains only a horizontal $\gamma_0$ and simple figure eight curves and for which $e_0^d = 2e_1^d = 2e_{-1}^d$ for each $\delta$-grading $d$: each height 1 or height $-1$ figure eight with grading $d$ contributes a pair $(a_i, b_i)$ with $\mrel(a_i) = -d$, $\mrel(b_i) = 1-d$, and $\Delta_{rel}(a_i) = \Delta_{rel}(b_i) = -1$, and each height 0 figure eight with grading $d$ contributes a pair $(a_i, b_i)$ with $\mrel(a_i) = -1-d$, $\mrel(b_i) = -d$, and $\Delta_{rel}(a_i) = \Delta_{rel}(b_i) = +1$.

If instead we consider $Y_\pm = S^3_{\pm 2}(9_{44})$, the computation is almost identical except that the generators now split into two spin$^c$ structures. We compute $\HFhat(Y_+)$ using part $(d)$ of Figure \ref{fig:9_44}, which shows the intersection of $\hatGamma(9_{44})$ with two lines of slope $+2$ (one for each spin$^c$ structure). Ignoring the generators coming form $\gamma_0$, we have $\HFhat_{red}(Y_+,0)$ generated by the four intersections with height zero figure eights (labeled $b_2, b_3, a_3$, and $a_2$ in Figure \ref{fig:9_44}),  which have relative gradings $\{0,0,-1,-1\}$ and $\Delta_{rel}$ of $1$, while $\HFhat_{red}(Y_+,1)$ is generated by the four intersections with the extremal figure eights (labeled $b_1$, $b_4$, $a_1$, and $a_4$ in Figure \ref{fig:9_44}), which have relative gradings $\{1,1,0,0\}$ and $\Delta_{rel}$ of $-1$. Thus $\HFhat(Y_+, 0)$ and $\HFhat(Y_-, 1)$ agree as relatively graded vector spaces, as do $\HFhat(Y_+, 1)$ and $\HFhat(Y_-, 0)$; in other words, any graded isomorphism from $\HFhat(Y_+)$ to $\HFhat(Y_+)$ must permute the two spin$^c$ structures. This could potentially conflict with the absolute grading, but it does not, since 
$$d(Y_+, 0) = d(L(2,1),0) = \tfrac 1 4 = d(L(2,-1),1) = d(Y_-,1), $$
$$d(Y_+, 1) = d(L(2,1),1) = -\tfrac 1 4 = d(L(2,-1),0) = d(Y_-,0).$$

Once again, the computation is essentially the same for any genus two knot for which $\hatGamma$ contains only a horizontal $\gamma_0$ and simple figure eight curves and for which $e_0^d = 2e_1^d = 2e_{-1}^d$ for each $\delta$-grading $d$.
There are other knots which satisfy this property; Table \ref{table:remaining-knots} gives 337 such knots. Surprisingly, these are the only knots the author is currently aware of for which Heegaard Floer homology does not obstruct all truly cosmetic surgeries. In particular, an example for which $\HFhat(S^3_{1/q}(K))$ agrees with $\HFhat(S^3_{-1/q}(K))$ for $q > 1$ or for a knot with $g(K) > 2$ has not yet been found. We note that is possible to construct a decorated immersed curve $\hatGamma$ which would allow for truly cosmetic surgeries with $g > 2$ or $q >1$, but is not known whether such curves occur as the invariant for a knot in $S^3$. To construct such a curve for some $g \ge 2$ and $q \ge 1$, we can place $q$ simple figure eights at height $g-1$ and, for symmetry, another $q$ simple eights at height $(1-g)$. Each of these figure eights produces $q$ pairs of points in the intersection with $\ell_{1,q}$. Indexing these pairs of points by $1 \le i \le q$, each pair has some relative gradings $m_i$ and $m_i+1$ and some grading shift $\Delta_{rel} = -\Delta_i < 0$. For each $i$, we then add $2\Delta_i$ simple figure eights at height 0, with two in each $\delta$-grading from $-m_i$ to $-m_i+\Delta_i-1$. Each of these figure eights at a $\delta$-grading $d$ produces a pair of intersection points with relative gradings $-d - 1$ and $-d$ and with $\Delta_{rel} = +1$; it is straightforward to check that these grading increases counteract the grading decreases from the extremal figure eights when $\phi$ is applied. Thus for the resulting graded multicurve $\hatGamma$, the Floer homologies $HF(\hatGamma, \ell_{1,q})$ and $HF(\hatGamma, \ell_{1,-q})$ agree as graded vector spaces.

\section{Computational results}\label{sec:computation}
One important consequence of Theorem \ref{thm:main} is that for any given knot $K$, cosmetic surgeries on $K$ are ruled out for all but possibly a finite number of pairs of slopes. In practice, cosmetic surgeries are obstructed outright for the vast majority of knots, and for the remaining knots the ``finite number" of possible pairs that need to be checked is quite small, often just two. Thus checking the cosmetic surgery conjecture on any finite set of knots reduces to distingushing a small number of pairs of manifolds. Computing, say, hyperbolic invariants for these remaining pairs is very tractable, and often this is sufficient to rule out the remaining surgeries. To demonstrate this, we check the following:

\begin{theorem}\label{thm:16-crossings}
The cosmetic surgery conjecture holds for all prime knots with at most 16 crossings.
\end{theorem}
\begin{proof}
We computed the $UV = 0$ knot Floer complex for all $\le 16$ crossing prime knots, using a program of Szab{\'o}, and then checked each against the obstructions described in this paper\footnote{All code used for these computations is available at \url{https://github.com/hanselman/CFK-immersed-curves}}. Recall that the $UV = 0$ knot Floer complex of $K$ is equivalent to the immersed curve invariant $\hatGamma(K)$. We make two observations from these computations:
\begin{itemize}
\item The maximum thickness of any prime knot up to 16 crossings is two; and
\item For each genus two knot up to 16 crossings, $\hatGamma(K)$ contains only simple figure eight components besides $\gamma_0$.
\end{itemize}
The first observation tells us immediately that we only need to consider genus two knots and we only need to consider the slopes $\pm 1$ and $\pm 2$. The second observation tells us that for these knots we can use the obstruction in Proposition \ref{prop:figure-eight-constraints-thin}.

There are $1,701,935$ knots up to 16 crossings. We note that the results of Ni and Wu (specifically conclusion $(i)$ in Theorem \ref{thm:NiWu}), already verify the conjecture for over two thirds of these knots: after restricting to knots with $\tau(K) = 0$, we are left with $449,417$ knots (requiring that $\epsilon = 0$ rather than $\tau = 0$ eliminates a further $38$ knots, leaving $449,379$). It turns out that the obstructions coming from Theorem \ref{thm:main} are much stronger. Among knots with $\epsilon(K) = 0$, requiring also that $g(K) = 2$ reduces the list to $3,316$. Finally, the obstruction in Proposition \ref{prop:figure-eight-constraints-thin} rules out truly cosmetic surgery on all but $337$ of these knots. The remaining knots are listed in Table \ref{table:remaining-knots}. Thus we have reduced to $674$ possible pairs of cosmetic surgeries, $\pm 1$ and $\pm 2$ surgeries on each of these 337 knots.  

This is the best that Heegaard Floer techniques alone can tell us; as noted in the previous section, for any knot $K$ satisfying the constraint in Proposition \ref{prop:figure-eight-constraints-thin}, Heegaard Floer homology cannot distinguish $S^3_{+1}(K)$ from $S^3_{-1}(K)$, nor can it distinguish $S^3_{+2}(K)$ from $S^3_{-2}(K)$. So these last examples must be ruled out using other methods. Computing the hyperbolic volume for the manifolds in question using SnapPy, we find that this distinguishes every pair except for the surgeries on four knots: $10_{33}$, $16n600112$, $16n786382$, and $16n988939$. These knots are amphichiral, so $+r$ surgery and $-r$ surgery can never be distinguished by hyperbolic volume. For these manifolds, the $\pm 1$ and $\pm 2$ surgery pairs on each of these four knots are distinguished by the Chern-Simons invariant, also computed by SnapPy\footnote{The author thanks Dave Futer for suggesting the use of the Chern-Simons invariant for amphichiral examples}. 
\end{proof}

\begin{remark}
We caution that Theorem \ref{thm:16-crossings} depends on computer calculations, some of which are non-verified. The computation of the knot Floer complex is combinatorial in nature, but SnapPy uses numerical methods to compute hyperbolic invariants so these computations should not be taken as rigorous proof. SnapPy does offer verified computation, via interval arithmetic, of hyperbolic volume but not of the Chern-Simons invariant. Thus, to be fully rigorous, the knots $10_{33}$, $16n600112$, $16n786382$, and $16n988939$ should be excluded from Theorem \ref{thm:16-crossings}. We do note that SnapPy \emph{estimates} the Chern-Simons computations are accurate to 10 decimal places, and for each relevant pair of manifolds the values differ by at least .001.
\end{remark}

\begin{remark}
Some of the 337 knots listed in Table \ref{table:remaining-knots} can be ruled out using the Jones polynomial as in \cite{IchiharaWu}, instead of using hyperbolic volumes. Unfortunately this does not help with any of the four knots for which the Chern-Simons invariant was needed.
\end{remark}

\begin{table}
\renewcommand{\arraystretch}{1.2} 
\small
\begin{tabular}{l l l l l l}
\midrule
9: & 41, 44 & 10: & 33, 136, 146 & 11a: & 333 \\
12a: & 1144 & 13a: & - & 14a: & 17464\\
15a: & 76589, 84220 \phantom{spacespacespa} & 16a: & 345268, 345454, 374264 \phantom{spacespacespa} & 11n: & 18, 42, 62, 83\\
\end{tabular}
\begin{tabular}{l l}
12n: & 34, 65, 278, 313, 360, 393, 430, 483, 550, 650, 846, 884\\
13n: & 71, 198, 490, 1019, 1209, 1398, 1513, 1598, 1756, 1757, 2337, 2703, 2796, 3290, 3416, 3783, 4591\\
14n: & 372, 971, 1193, 2087, 2489, 6421, 7228, 7412, 7469, 7534, 8091, 8196, 8554, 8716, 9290, 9684, 9829, \\
 & 10155, 11129, 11429, 12224, 12609, 12977, 13570, 14799, 15285, 15380, 15581, 15965, 15976, 17163, \\
 & 17183, 18494, 19673, 21231, 21269, 22150, 22196, 22614, 22634, 23325, 24593, 27072, 27091\\
15n: & 1058, 3240, 4898, 9477, 11491, 19192, 21666, 21997, 27824, 30711, 34041, 34773, 36113, 38567, \\
& 38594, 41604, 43982, 46350, 46536, 49081, 51379, 51847, 54458, 58840, 62260, 63468, 63550, 64468, \\
& 67694, 67879, 71170, 73390, 73507, 76978, 77245, 77247, 77784, 83761, 84434, 84645, 88899, 91448, \\
& 93899, 94474, 96914, 97157, 102309, 104775, 105829, 106611, 118711, 120250, 124511, 129229, \\
& 129231, 132539, 135706, 137623, 140373, 140582, 142082, 142299, 142716, 142841, 142843, 143482, \\
& 143825, 143856, 144436, 144439, 144887, 147186, 156806, 160027\\
16n: & 5596, 9193, 16004, 24365, 27992,49009, 60136, 67523, 94939, 102539, 102773, 191694, 196472, \\
& 197735, 203049, 215168, 218032, 219174, 220556, 227624, 230857, 233335, 239267, 239379, 242042, \\
& 242545, 249927, 265957, 271606, 271610, 273164, 277974, 280482, 285128, 306917, 307635, 315594, \\
& 324571, 329529, 332372, 349983, 353272, 360174, 366612, 376208, 385669, 386732, 387806, 401152, \\
& 401963, 402644, 405088, 412371, 423420, 424451, 429723, 438719, 440479, 441595, 459035, 460502, \\
& 461585, 463225, 463419, 465019, 466470, 467558, 469510, 470606, 470717, 473737, 475444, 481843, \\
& 489486, 493489, 494163, 498542, 498651, 508893, 513585, 515663, 534392, 540621, 544661, 550305, \\
& 551107, 577882, 585135, 587843, 588588, 596192, 596449, 597513, 598535, 599034, 600112, 606009, \\
& 608181, 609311, 609798, 614804, 614907, 617672, 628265, 629526, 631987, 632225, 635338, 666646, \\
& 687419, 691300, 696924, 696992, 725574, 761555, 762559, 767010, 768788, 770126, 774829, 784110, \\
& 786382, 788898, 789181, 798964, 809799, 810368, 812243, 824554, 828723, 847911, 855704, 855909, \\
& 862009, 863179, 864017, 864258, 864259, 869439, 869441, 874997, 879694, 880152, 888060, 888954, \\
& 902353, 906603, 907441, 907673, 916183, 916207, 916242, 918157, 919068, 925408, 932460, 941562, \\
& 941564, 968742, 972142, 988939, 989795, 990225, 990270, 991069, 991085, 998071, 1000650, 1000651, \\
& 1001406, 1001474, 1004278, 1004646 \\
\midrule
\end{tabular}
\caption{Knots up to 16 crossing for which Heegaard Floer homology does not rule out all cosmetic surgeries. The only pairs of slopes not ruled out for each of these knots are $\pm 1$ and $\pm 2$.}
\label{table:remaining-knots}
\end{table} 

The result above considers prime knots, but Theorem \ref{thm:main} is also very good at obstructing truly cosmetic surgeries on connected sums. In fact, with only a little more work, we can rule out cosmetic surgeries on all knots whose prime summands have at most 16 crossings. The following is equivalent to Theorem \ref{thm:16-crossings-connected-sums} stated in the introduction.
\begin{theorem}
The cosmetic surgery conjecture holds for the connected sum of any number of prime knots each with $\le 16$ crossings.
\end{theorem}
\begin{proof}
Suppose that $K$ has $n > 1$ prime summands, each with at most 16 crossings. If follows that $g(K) \ge n$ and that $th(K) \le 2n$, since both genus and thickness are additive with respect to connected sum and the maximum thickness for knots up to 16 crossings is two. Suppose $\pm r$ is a pair of truly cosmetic surgery slopes for $K$. If $g(K) > 2$, then by Theorem \ref{thm:main} we must have $r = 1/q$ with
$$q \le \frac{th(K)}{2g(K)(g(K)-1)} + \frac{1}{g(K)-1} \le \frac{2n}{2n(n-1)} + \frac{1}{n-1} = \frac{2}{n-1}.$$
Since $q$ must be $\ge 1$, it follows that $n \le 3$. 

Moreover, if $n = 3$ then $g(K) = 3$ and $q=1$. If $n = 2$ then $g(K) = 2$ but $q$ can be 1 or 2; we also must consider the case that $g(K) = 2$ and $r = 2$. By Proposition \ref{prop:ns-ratio}, if $n = 3$ we require that $e_0 = 2e_1 +7e_2$, while if $n=2$ we require $e_0 = 2e_1$ if $r \in \{1,2\}$ and $e_0 = 6e_1$ if $r = 1/2$.

We will need one more observation from our computations of $\hatGamma$:
\begin{itemize}
\item For knots with $\le 16$ crossings and $g(K) = 1$, all curves other than $\gamma_0$ are simple figure eights, and $\gamma_0$ has one of the three forms: the horizontal curve, the curve which is the invariant of the right-handed trefoil (see Figure \ref{fig:curve-examples}), or the mirror of the right-handed trefoil curve (which is the invariant of the left-handed trefoil). 
\end{itemize}
We will denote the three possibilities for $\gamma_0$ above as $\gamma_0^{horiz}$, $\gamma_0^{RHT}$, and $\gamma_0^{LHT}$. Note that they are distinguished by the value of $\tau$, which is 0, 1, and $-1$, respectively.

If $n = 2$, we must have $K = K_1 \# K_2$ with $g(K_1) = g(K_2) = 1$. By the Kunneth formula for knot Floer homology, $\hatGamma(K) = \hatGamma(K_1) \otimes \hatGamma(K_2)$\footnote{This is an abuse of notation, we really mean the tensor product of the corresponding bifiltered chain complexes; an immersed curve description of the tensor product operation is given in \cite[Section 4]{HRW:companion}}. A straightforward computation reveals that the tensor product of two figure eight components (at height 0) yields four figure eight components at heights $1,0,0$, and $-1$. The tensor product of a figure eight component and either $\gamma_0^{RHT}$ or $\gamma_0^{LHT}$ gives three figure eight components at heights $-1, 0,$ and $1$, and the tensor product of $\gamma_0^{RHT}$ with $\gamma_0^{LHT}$ yields a horizontal curve $\gamma_0$ along with two figure eights at heights $-1$ and $1$. Taking the tensor product of any curve with $\gamma_0^{horiz}$ gives a copy of that curve. Let $a$ and $b$ be the number of height one figure eight components in $\hatGamma(K_1)$ and $\hatGamma(K_2)$. Since $\tau$ is additive and $\tau(K) = 0$, the $\gamma_0$ curves for $K_1$ and $K_2$ either are $\gamma_0^{RHT}$ and $\gamma_0^{LHT}$ (in that order, without loss of generality) or they are both horizontal and $a,b > 0$. In the first case, we see that $\hatGamma(K)$ has $e_0 = 2ab + a + b$ and $e_1 = ab + a + b + 1$; this is impossible since then $e_0 < 2e_1$. In the second case we have $e_0 = 2ab + a + b$ and $e_1 = ab$; this is impossible because $\tfrac{e_0}{e_1} = 2 + \tfrac{1}{a} + \tfrac{1}{b}$ is strictly greater than 2 and strictly less than 6. Thus $n \neq 2$.

If $n = 3$, we have $K = K_1 \# K_2 \# K_3$ with $g(K_i) = 1$. Let $a$, $b$, and $c$ be the number of figure eight components in $\hatGamma$ for $K_1$, $K_2$, and $K_3$. Since $\tau(K) = 0$, either $\gamma_0$ for the three knots is given (up to reordering) by $\gamma_0^{RHT}$, $\gamma_0^{LHT}$, and $\gamma_0^{horiz}$, or $\gamma_0$ is horizontal for all three knots. In the first case, we compute that $e_0 = 6abc + 4ac + 4bc + 3c +2ab + a + b$, and that $e_1 = 4abc + 3ac+3bc+2c + ab + a + b + 1$. It follows that $2e_1 - e_0 = 2abc + 2ac + 2bc + c + a + b + 2 > 0$, which is a contradiction because we require $e_ 0 = 2e_1 + 7e_2 \ge 2e_1$. In the second case, we compute that $e_0 = 4abc + 2(ab+ac+bc) + a+b+c$, $e_1 = 3abc + (ab+ac+bc)$, and $e_2 = abc$. It follows that $7e_2 + 2e_1 > e_0$, since $7e_2 + 2e_1 - e_0 = 9abc - (a+b+c)$ is strictly positive. Thus $n \neq 3$.
\end{proof} 

\begin{remark}
Note that in the above proof, the thickness bound immediately ruled out connected sums of more than three 16 crossing knots. Since $th(K)$ and $g(K)$ are both additive and the upper bound on $q$ goes like $th/g^2$, this behavior is expected. In fact, for any finite set of knots Theorem \ref{thm:main} prohibits truly cosmetic surgeries on connected sums of sufficiently many knots in the initial set.
\end{remark}

\bibliographystyle{plain}
\bibliography{bibliography}

\end{document}